\documentclass[a4paper,twoside]{article}
\usepackage{a4}

\usepackage{hyperref}     
\hypersetup{
	colorlinks=true,        
	linkcolor=blue,         
	citecolor=blue,
	urlcolor=blue,
}

\usepackage{graphicx}
\usepackage{amssymb}
\usepackage{epstopdf}
\usepackage{amsmath,amsthm}
\numberwithin{equation}{section}
\numberwithin{figure}{section}

\usepackage{upref}
\usepackage[active]{srcltx}
\usepackage{xypic}

\usepackage{amsthm,amsmath,amssymb}

\usepackage{mathrsfs}

\usepackage{graphicx} 
\usepackage{amsmath, bm}

\theoremstyle{plain}
\newtheorem{thm}{Theorem}[section]

\newtheorem{lem}[thm]{Lemma}
\newtheorem{cor}[thm]{Corollary}

\newtheorem{definition}[thm]{Definition}

\newtheorem{remark}[thm]{Remark}

\newcommand{\R}{\mathbb{R}}

\newcommand{\norm}[1]{\left\| #1 \right\|}
\newcommand{\abs}[1]{\left| #1 \right|}
\newcommand{\dint}{\iint_{\R^{n}\times\R^{n}}}

\begin{document}
	\title{Korevaar-Schoen Energy and Interpolations of Fractional Sobolev mappings on Carnot groups}
	\author{Yihan Cui}
	\maketitle
	\begin{abstract}
		This paper investigates the theory of fractional Sobolev mappings between Carnot groups, with a particular focus on the Korevaar-Schoen energy functional and interpolation properties. We establish that weak contact equations hold for fractional Sobolev mappings on Carnot groups, generalizing classical results from the Euclidean setting to the sub-Riemannian framework. Our main results include: (1) the proof that fractional Sobolev mappings satisfying weak contact equations belong to the horizontal fractional Sobolev space $W_H^{s-\lambda,q}(\Omega;G_2)$
		for any sufficiently small $\lambda>0$; (2) interpolation theorems for fractional Sobolev mappings between Carnot groups, demonstrating that the limiting behavior as $s \to 1$
		recovers the first-order Sobolev space; (3) a rigorous treatment of the Korevaar-Schoen energy on Carnot groups, including a characterization of its minimizers via horizontal differential forms. 
	\end{abstract}
	\tableofcontents

	\section{Introduction}
	The calculus of variations for mappings taking values in a non-smooth
	target has its origin in the theory of harmonic maps. After the
	foundational work of Eells and Sampson \cite{eells1964harmonic} for
	smooth Riemannian targets, Korevaar and Schoen
	\cite{korevaar1993sobolev,korevaar1997global} initiated the study of
	Sobolev mappings $u:\Omega\to X$ from a Riemannian domain into an
	arbitrary metric space $(X,d)$. Their starting point is that, when the
	target carries no linear structure, the weak derivative of $u$ has to
	be replaced by an \emph{energy} which measures, at scale
	$\varepsilon$, the average distortion of pairs of values:
	\begin{equation}\label{eq:KS-classical}
		E_\varepsilon(u) = \int_{\Omega} \frac{1}{|B(x,\varepsilon)|}
		\int_{B(x,\varepsilon)}
		\frac{d\big(u(x),u(y)\big)^{2}}{\varepsilon^{2}}\,dy\,dx ,
	\end{equation}
	the Sobolev class $W^{1,2}(\Omega;X)$ being defined by the condition
	$\sup_{\varepsilon>0}E_\varepsilon(u)<\infty$. This idea has proved
	remarkably fertile: it yields existence and regularity theories for
	harmonic maps into singular and non-locally compact targets
	\cite{korevaar1993sobolev,korevaar1997global,eells2001harmonic,jost1997nonpositive},
	and it is one of several equivalent ways of defining a first-order
	Sobolev space of metric-space valued maps --- the others being based
	on post-composition with Lipschitz functions and upper gradients
	(Reshetnyak \cite{reshetnyak1997sobolev,reshetnyak2004sobolev}), on
	pointwise inequalities with integrable ``gradients'' in the sense of
	Haj{\l}asz
	\cite{hajlasz1996sobolev,hajlasz2000sobolev}, and on the modulus of
	curve families (Newtonian spaces, Shanmugalingam
	\cite{shanmugalingam2000newtonian}). These approaches are known to be
	compatible under the now customary assumption that the source is a
	doubling metric measure space supporting a Poincar\'e inequality, an
	investigation which originated in Cheeger's measurable differentiable
	structure \cite{cheeger1999differentiability} and is systematically
	exposed in
	\cite{hajlasz2000sobolev,shanmugalingam2000newtonian,ambrosio2005gradient}.
	
	A second, more recent line of research concerns the \emph{variational
		limit} of functionals of the type \eqref{eq:KS-classical}. For $p>1$
	the $\Gamma$-limit of the Korevaar--Schoen $p$-energies on a Cheeger
	space has been identified recently by Alonso Ruiz and Baudoin
	\cite{ruiz2024korevaarschoenpenergiesgammalimitscheeger}: the limit is a constant multiple of the
	Cheeger $p$-energy, whose domain is the Newtonian space $N^{1,p}$,
	while for $p=1$ it is a total-variation functional on $BV$; on compact
	spaces the convergence even improves to Mosco convergence. Since a
	Carnot group endowed with its Carnot--Carath\'eodory (CC) distance, a
	Haar measure and the horizontal distribution is a prototypical example
	of a Cheeger space, this circle of ideas provides the natural
	variational framework for the energy studied in the present paper.

	The geometric setting of this paper is that of \emph{Carnot groups},
	i.e.\ simply connected nilpotent Lie groups whose Lie algebra admits a
	stratification $\mathfrak g=\mathfrak g^{[1]}\oplus\cdots\oplus
	\mathfrak g^{[s]}$ with $[\mathfrak
	g^{[1]},\mathfrak g^{[k]}]=\mathfrak g^{[k+1]}$. These groups carry a
	left-invariant sub-Riemannian (CC) metric $d_{cc}$, a one-parameter
	family of dilations $\delta_\lambda$ with respect to which $d_{cc}$ is
	$1$-homogeneous, and a homogeneous dimension
	$Q_G=\sum_k k\dim\mathfrak g^{[k]}$ which governs all
	scaling-invariant inequalities. Their analysis has been developed since
	the 1980s; see e.g.\
	\cite{folland1982hardy,garofalo1996isoperimetric}.
	
	What makes analysis on Carnot groups genuinely different from the
	Euclidean theory is the correct notion of derivative. By the
	celebrated theorem of Pansu \cite{pansu1989metriques}, a Lipschitz map
	between Carnot groups is differentiable almost everywhere in the
	following \emph{dilated} sense: the limit of
	$\delta_{1/t}\big(f(p)^{-1}f(p\,\delta_t\xi)\big)$ as $t\to0$ exists
	and is a \emph{group homomorphism} which, moreover, respects the
	stratification. Thus the Pansu differential is not merely a linear map,
	but a homomorphism of stratified Lie algebras, and this algebraic
	rigidity has a decisive analytic consequence: if $X$ is a horizontal
	vector field and $\omega$ a contact form belonging to a
	\emph{different} layer of the target, then $X(f^{*}\omega)=0$. In
	other words, Pansu differentiable maps automatically satisfy a system
	of \emph{contact equations}. Pansu's theorem has been extended in
	several directions; in particular Vodopyanov
	\cite{Vodopyanov2003DifferentiabilityOM} proved that weakly contact
	Sobolev maps of class $W^{1,p}$ with $p$ above the homogeneous
	dimension are Pansu differentiable almost everywhere, together with
	the accompanying change-of-variables formulas and Lusin-type
	properties, and the interplay between contact geometry, Pansu
	differentiability and area/coarea formulas has been further clarified
	in
	\cite{magnani2005coarea,magnani2013towards,warhurst2008contact,vodopyanov2007traces}.
	
	At the level of first-order Sobolev mappings, the contact condition is
	the defining constraint of the variational theory of
	\emph{Legendrian} (horizontal) maps. Capogna and Lin
	\cite{Capogna2001} introduced and studied Legendrian energy
	minimizers with Heisenberg-group target, and Tan \cite{tan2007sobolevclasseshorizontalenergy}
	developed the theory of Sobolev classes and horizontal energy
	minimizers between Carnot--Carath\'eodory spaces. A recurrent
	phenomenon in all these works is that the \emph{vertical}
	(higher-layer) components of a contact map are not free: in a step-two
	group they are determined by the horizontal ones through a quadratic
	relation, and this is precisely the mechanism that makes the
	corresponding variational problems tractable.

	Parallel to these developments, the theory of fractional Sobolev
	spaces has undergone a spectacular expansion. The classical
	Gagliardo--Slobodeckij spaces $W^{s,q}$, $0<s<1$, based on the double
	integral
	\[
	[f]_{W^{s,q}}^q
	= \int\!\!\int \frac{|f(x)-f(y)|^{q}}{|x-y|^{n+sq}}\,dx\,dy ,
	\]
	are by now a standard tool in the calculus of variations and in
	nonlinear PDEs \cite{dinezza2012hitchhiker}. A landmark result of
	Bourgain, Br\'ezis and Mironescu \cite{bourgain2001another} describes
	the endpoint behaviour as $s\to1$: for $f\in W^{1,q}(\mathbb R^{n})$,
	\[
	\lim_{s\to1}(1-s)\int\!\!\int
	\frac{|f(x)-f(y)|^{q}}{|x-y|^{n+sq}}\,dx\,dy
	= C_{n,q}\int_{\mathbb R^{n}} |\nabla f|^{q} ,
	\]
	so that the first-order Sobolev seminorm is recovered as the limiting
	case of the fractional ones; the companion endpoint $s\to0$ was
	obtained by Maz'ya and Shaposhnikova \cite{mazya2002bourgain}.
	
	Both statements have been transported to Carnot groups, where the
	Euclidean distance in the Gagliardo kernel is replaced by the CC
	distance and the exponent $n$ by the homogeneous dimension. Maalaoui
	and Pinamonti \cite{MAALAOUI201991} used real
	interpolation (the $K$-method) to characterize fractional Sobolev
	spaces on Carnot groups and proved the corresponding
	Bourgain--Br\'ezis--Mironescu and Maz'ya--Shaposhnikova type endpoint
	formulas, together with the interpolation identity
	$(L^{q},\dot W^{1,q})_{s,q}=W^{s,q}$ that we use repeatedly below. The
	embedding theory for fractional Slobodeckij spaces on metric-measure
	spaces, and their relation with the fractional Haj{\l}asz spaces
	$M^{s,q}$, was developed by G\'orka and S{\l}abuszewski
	\cite{GORKA2022112867}. On nilpotent Lie groups, function spaces of
	Besov and Sobolev type have an older and independent tradition going
	back to Saka \cite{Saka1979}, based on representation theory and
	difference operators; we refer to it for background, but not for the
	contact phenomena studied here.

	The first-order contact theory recalled above rests on Pansu
	differentiability, which is available for Lipschitz maps and, in the
	Sobolev setting, for sufficiently integrable weakly contact maps. For a
	merely H\"older continuous map, and \emph{a fortiori} for a fractional
	Sobolev map $f\in W^{s,q}_H(\Omega;G_2)$ with $0<s<1$, a Pansu
	differential need not exist at any point. Consequently, the following
	questions --- all of which are classical and well understood in the
	first-order case --- are, to the best of our knowledge, open in the
	fractional setting:
	
	\begin{enumerate}
		\item[(a)] Is there a meaningful \emph{contact condition} for
		fractional Sobolev mappings into a Carnot group, given that no
		pointwise Pansu differential is available?
		
		\item[(b)] Do the fractional seminorms of such a map form an
		interpolation scale whose endpoint $s\to1$ recovers the horizontal
		first-order energy $\|D_Hf\|_{L^{q}}$?
		
		\item[(c)] How does the Korevaar--Schoen energy of a fractional
		Sobolev map decompose into a horizontal part and a vertical defect,
		and which equation is satisfied by its minimizers?
		
		\item[(d)] Can one define, in a distributional sense, the Jacobian
		determinant of a Sobolev map with Heisenberg-group target, in the
		spirit of the Euclidean theory of distributional Jacobians
		\cite{GLADBACH2020108312,tmna/1475178877,Fausto,Li2021}?
	\end{enumerate}
	
	The purpose of this paper is to answer these four questions. Our main
	tools are group mollification $f_\varepsilon=f*\eta_\varepsilon$ on the
	source Carnot group, the stratified Taylor formula of Arena, Caruso and
	Causa \cite{Arena2010}, Haj{\l}asz upper gradients, and the
	interpolation results of \cite{MAALAOUI201991}.
	
	Throughout, $\Omega\subset G_1$ is a domain in a Carnot group, $G_2$ is
	another Carnot group with CC distance $d_{cc}$ and topological
	dimension $\nu_2$, and $W^{s,q}_H(\Omega;G_2)$ denotes the space of
	maps $f$ with
	\[
	\int_{\Omega}\int_{\Omega}
	\frac{d_{cc}\big(f(x),f(y)\big)^{q}}
	{d_{cc}(x,y)^{\nu_1+sq}}\,dx\,dy < \infty ,
	\]
	where $\nu_1$ is the homogeneous dimension of $G_1$.
	
	\paragraph{(I) Weak contact equations for fractional Sobolev mappings
		(Section~\ref{sec:contact}).}
	Our first result, Theorem~\ref{thm:contact-direct}, shows that, in
	spite of the absence of a pointwise Pansu differential, fractional
	Sobolev mappings satisfy a \emph{weak} contact equation. For
	$f\in W^{s,q}_H(\Omega;G_2)$ with $s>1/2$ and $q\ge2$, and for every
	horizontal left-invariant vector field $X$,
	\begin{equation}\label{eq:intro-contact}
		X\big(f_\varepsilon^{*}\omega\big)\longrightarrow 0
		\quad\text{in } L^{\frac q2}_{\mathrm{loc}}(\Omega;\mathbb R^{\nu_2})
		\qquad\text{as }\varepsilon\to0 ,
	\end{equation}
	where $f_\varepsilon=f*\eta_\varepsilon$ is the group mollification and
	$\omega$ is any contact form of the second layer; moreover
	$\pi f_\varepsilon\to\pi f$ in $W^{s,q}(\Omega)$, where $\pi$ denotes
	the projection onto the horizontal layer. 
	
	For $q=2$ the proof is a
	direct use of Fubini theorem, while for $q>2$ we use the
	observation that $s>1/2$ permits one to work with $s-\lambda$ in place
	of $s$ for arbitrarily small $\lambda>0$, which produces the extra
	integrability needed.
	
	Our second result, Theorem~\ref{thm:contact-converse}, is a converse,
	and is perhaps the more interesting of the two: \emph{if} the
	horizontal projection $\pi f$ belongs to
	$W^{s,q}(\Omega;\mathbb R^{\nu_2})$ \emph{and} $f$ satisfies the weak
	contact equations, \emph{then}
	$f\in W^{s-\lambda,q}_H(\Omega;G_2)$ for every $\lambda>0$
	sufficiently small. The proof combines the Haj{\l}asz gradient
	characterization of fractional spaces (available on Carnot groups by
	\cite{GORKA2022112867}) with the quadratic control of the vertical
	components,
	\[
	|f^{B_k}(x)-f^{B_k}(y)|\lesssim |\pi f(x)-\pi f(y)|^{2} ,
	\]
	which is the distributional trace of the commutator structure of
	$G_2$. Thus, up to an arbitrarily small loss of differentiability, the
	full group-valued fractional seminorm of a contact map is controlled
	by its horizontal projection. This is the fractional counterpart of
	the classical fact that a contact map is determined by its horizontal
	components.
	
	\paragraph{(II) Interpolation (Section~\ref{sec:interp}).}
	Theorem~\ref{thm:interp-limit} is the
	Bourgain--Br\'ezis--Mironescu phenomenon for mappings between Carnot
	groups: if $f\in\bigcap_{s\in(0,1)}\dot W^{s,q}_H(\Omega;G_2)$ on a
	bounded domain $\Omega$ and
	\[
	\lim_{s\to1}(1-s)^{\frac2q}
	\|f\|^{\frac q2}_{W^{s,\frac q2}_H(\Omega;G_2)} < \infty ,
	\]
	then
	\[
	\|D_Hf\|_{L^{\frac q2}}^{\frac q2}
	\lesssim
	\lim_{s\to1}(1-s)^{\frac2q}
	\|f\|^{\frac q2}_{W^{s,\frac q2}_H(\Omega;G_2)} ,
	\]
	and consequently $f\in W^{1,p}_H(\Omega;G_2)$ with
	$D_Hf:=D_H(\pi f)$. The strategy is to reduce the assertion to the
	scalar interpolation theorem of Maalaoui and Pinamonti
	\cite[Theorem~5.1]{MAALAOUI201991} applied to $\pi f$, the
	contact part being supplied by Theorem~\ref{thm:contact-direct}. The
	converse inclusion
	$W^{1,q}_H(\Omega;G_2)\subset\bigcap_{s\in(0,1)}\dot
	W^{s,q}_H(\Omega;G_2)$, with the natural estimate, is
	Theorem~\ref{thm:interp-converse} and follows from the Poincar\'e
	inequality. 
	
	\paragraph{(III) The Korevaar--Schoen energy on Carnot groups
		(Section~\ref{sec:energy}).}
	For $f\in W^{s,q}_H(\Omega;G_2)$ we define (cf.\ Definition~\ref{def:KS-energy})
	\[
	E_{s,q}(f)=\limsup_{r\to0}\frac{1}{r^{\nu_1+sq}}
	\int_{\Omega}\int_{B(x,r)}
	d_K\big(f(x),f(y)\big)^{q}\,dy\,dx ,
	\]
	where $d_K$ is the layer-adapted quasi-metric, comparable to $d_{cc}$.
	The key technical step is Lemma~\ref{lem:taylor}, a stratified
	expansion in the spirit of
	\cite[Theorem~2.13]{Capogna2001} and
	\cite{Arena2010}, which identifies the leading vertical term:
	\[
	\frac{\varphi^{k}\big(f(p),f(\tilde p)\big)}
	{d_K(p,\tilde p)^{2s}}
	=
	\Big\langle \sum_j T_j f^{*}(\omega^{k})(p),\,
	\frac{\varphi(p,\tilde p)}
	{\big(|\pi(p)-\pi(\tilde p)|^{2}
		+\sum_l|\varphi^{l}(p,\tilde p)|\big)^{s}}
	\Big\rangle
	+o\big(d_K(p,\tilde p)^{2-2s}\big)
	\]
	in $L^{q/2}(\Omega)$. From this we obtain
	(Theorem~\ref{thm:energy}) the following dichotomy, which is the main
	result of Section~\ref{sec:energy}:
	
	\begin{itemize}
		\item for $0<s<1$, the vertical part disappears in the limit and the
		energy depends only on the horizontal projection:
		\[
		E_{s,q}(f)=\limsup_{r\to0}\frac{1}{r^{\nu_1+sq}}
		\int_{\Omega}\int_{B(x,r)}
		|\pi f(x)-\pi f(y)|^{q}\,dy\,dx ;
		\]
		
		\item for $s=1$, a vertical \emph{defect} survives:
		\[
		E_{1,q}(f)\ \ge\ \limsup_{r\to0}\frac{1}{r^{\nu_1+q}}
		\int_{\Omega}\int_{B(x,r)}
		|\pi f(x)-\pi f(y)|^{q}\,dy\,dx
		\;+\;
		C\sum_k\big\|T f^{*}(\omega^{k})\big\|_{L^{\frac q2}}^{\frac q2} .
		\]
	\end{itemize}
	
	As a consequence
	(Corollary~\ref{cor:vertical}), for $f\in W^{1,q}_H(\Omega;G_2)$ these
	vertical gradients are controlled by $\|D_Hf\|_{L^{q}}$, and the
	second-layer components of $f$ are genuinely Sobolev-regular in the
	vertical directions.
	
	We then turn to the variational problem. Since Carnot groups are
	Cheeger spaces, the $\Gamma$-convergence machinery of Alonso Ruiz and
	Baudoin \cite[Corollary~3.13]{ruiz2024korevaarschoenpenergiesgammalimitscheeger} applies, and we show
	that a minimizer of $E_{s,q}$ satisfies the nonlocal
	Euler--Lagrange equation
	\[
	(\Delta_q)^{s}_{KS}\,\tilde g
	=\lim_{r\to0}\frac{2}{r^{\nu_1+sq}}
	\int_{\Omega}\int_{B(x,r)}
	|\tilde g(x)-\tilde g(y)|^{q-2}
	\big(\tilde g(x)-\tilde g(y)\big)\,dy\,dx = 0 ,
	\]
	a fractional $q$-Laplace equation of Korevaar--Schoen type (cf.\
	Definition~\ref{def:frac-q-laplace}). Combined with Theorem~\ref{thm:energy}, this reduces
	the minimization of $E_{s,q}$ on $W^{s,q}_H(\Omega;G_2)$ to a
	minimization problem for the horizontal components.
	
	\paragraph{(IV) The Jacobian of fractional Sobolev mappings
		(Section~\ref{sec:jacobian}).}
	Finally, we show that the distributional Jacobian determinant can not be
	defined for fractional Sobolev mappings $f\in W_H^{s,q}(\Omega;G_2)$ with the method on Euclidean space \cite{hajłasz2025holdercontinuousmappingsdifferential}, when $0<s<1$. 
	
	\section{Preliminary}
	We recall the basic properties of Carnot groups. 
	A Lie algebra  $\mathfrak{g}$  is stratified if it has a decomposition  $\mathfrak{g}=\mathfrak{g}^{[1]} \oplus \cdots \oplus \mathfrak{g}^{[s]}$  with  $\left[\mathfrak{g}^{[1]}, \mathfrak{g}^{[k]}\right]=\mathfrak{g}^{[k+1]}$  for all $ k \geq 1$ , where we denote  $\mathfrak{g}^{[k]}=\{0\}$  for  $k \geq s+1 $. A (sub-Riemannian) Carnot group is a simply connected nilpotent Lie group $ G$  with a stratified Lie algebra  $\mathfrak{g}$  equipped with an inner product on the horizontal layer  $\mathfrak{g}^{[1]}$ . For convenience, We will need an inner product not only on the horizontal layer  $\mathfrak{g}^{[1]}$ , but instead on the whole Lie algebra  $\mathfrak{g}$ . We always equip the Lie algebra of a Carnot group with an inner product for which the layers  $\mathfrak{g}^{[1]} \oplus \cdots \oplus \mathfrak{g}^{[s]}$  are pairwise orthogonal. The homogeneous dimension of $ G$  is  $Q_{G}:=\sum_{k=1}^{s} k \operatorname{dim}\left(\mathfrak{g}^{[k]}\right) $. The (horizontal) rank of  $G$  is the dimension of the horizontal layer  $\mathfrak{g}^{[1]} $.
	
	Given $ g \in G$ , we denote by  $L_{g}: G \rightarrow G$  the left-translation  $L_{g}(h)= g h$ . We identify  $T_{e} G$  with  $\mathfrak{g}$ . An absolutely continuous curve  $\gamma:[0,1] \rightarrow G$  is horizontal if for almost every  $t \in[0,1] $ its left-trivialized derivative  $\left(L_{\gamma(t)}^{-1}\right) * \dot{\gamma}(t)$  is contained in the horizontal layer $ \mathfrak{g}^{[1]} $. The length of the horizontal curve  $\gamma$  is
	\begin{equation*}
		\ell(\gamma)=\int_{0}^{1}\left\|\left(L_{\gamma(t)}^{-1}\right)_{*} \dot{\gamma}(t)\right\| d t
	\end{equation*}
	where the norm is the one induced by the inner product on  $\mathfrak{g}^{[1]}$ . The sub-Riemannian distance between two points  $g, h \in G $ is
	\begin{equation*}
		d(g, h)=\inf \{\ell(\gamma) \mid \gamma:[0,1] \rightarrow G \text { horizontal, } \gamma(0)=g, \gamma(1)=h\} .
	\end{equation*}
	
	By construction, this distance is left-invariant.
	Carnot groups also admit a one-parameter group of automorphisms, known as the dilations  $\delta_{\lambda}: G \rightarrow G, \lambda>0 $. The associated Lie algebra automorphisms, also denoted  $\delta_{\lambda}: \mathfrak{g} \rightarrow \mathfrak{g} $, are the linear maps defined on each layer of the stratification as
	\begin{equation*}
		\delta_{\lambda}(X)=\lambda^{k} X, \quad X \in \mathfrak{g}^{[k]}, \quad 1 \leq k \leq s .
	\end{equation*}
	The sub-Riemannian distance is 1-homogeneous with respect to these dilations, i.e.,  $d\left(\delta_{\lambda} g, \delta_{\lambda} h\right)=\lambda d(g, h) $.
	
	In conventions, all Carnot groups are equipped with a left-invariant metric and a Haar
	measure $\mathcal H^{\nu_1}$, with respect to which all integrals are
	taken. We write $A\lesssim B$ if $A\le CB$ for a constant $C$
	depending only on the groups, the domain and the exponents, and
	$A\approx B$ if $A\lesssim B\lesssim A$.
	
	\begin{definition}
		Denote $W_H^{s,q}(\Omega;G_2)$ as the fraction Sobolev mappings space respected to $C-C$ metric on Carnot group $G_2$
		\begin{equation*}
			\int_{\Omega}\int_{\Omega}\frac{d_{cc}(f(x),f(y))^q}{d_{cc}(x,y)^{\nu_1+sq}}dxdy <\infty
		\end{equation*} 
		for any $f\in W_H^{s,q}(\Omega;G_2)$, where $G_1$ is a Carnot group and $\Omega\subset G_1$.
	\end{definition}
	Because in this paper we would deeply take advantage of the functions on Carnot groups, it is necessary to define the fraction Sobolev functions space on Carnot groups.
	\begin{definition}
		Denote $W^{s,q}(\Omega;\mathbb{R}^{n_2})$ as the fraction Sobolev mappings space respected to Euclidean metric on $\mathbb{R}^{n_2}$
		\begin{equation*}
			\int_{\Omega}\int_{\Omega}\frac{|f(x)-f(y)|^q}{d_{cc}(x,y)^{\nu_1+sq}}dxdy <\infty
		\end{equation*} 
		for any $f\in W^{s,q}(\Omega;G_2\mathbb{R}^{n_2})$, where $G_1$ is a Carnot group and $\Omega\subset G_1$.
	\end{definition}

	\section{Weak contact equation holds for fractional Sobolev mappings on Carnot groups}
	\label{sec:contact}
	The purpose of this section is to answer question (a) of the introduction: we
	explain in what sense a merely fractional Sobolev mapping with values in a
	Carnot group satisfies a \emph{contact} equation, even though no pointwise Pansu
	differential is available.
	
	The section is organised as follows. We first fix explicit coordinates on a
	step-three target group, in which the group law, the left-invariant frame, the
	contact forms and the layer-adapted quasi-metric $d_K$ can all be written down.
	This choice is made for the sake of readability rather than out of necessity:
	every estimate below uses only the stratification and the resulting homogeneity
	of $d_K$, and the same computation applies verbatim to groups of arbitrary
	step. Against this background we then prove two results pointing in opposite
	directions.
	\begin{itemize}
		\item Theorem~\ref{thm:contact-direct} (the direct statement) says that,
		after mollification, the contact defect $X(f_\varepsilon^{*}\omega)$ becomes
		an integrable quantity which vanishes as $\varepsilon\to0$; in other words,
		fractional Sobolev mappings are \emph{asymptotically contact} in exactly the
		sense in which Pansu differentiable maps are contact.
		\item Theorem~\ref{thm:contact-converse} (the converse) says that the weak
		contact equation, combined with the $W^{s,q}$ regularity of the horizontal
		projection $\pi f$, already controls the full group-valued seminorm, up to an
		arbitrarily small loss $\lambda>0$ of differentiability.
	\end{itemize}
	Taken together, these two statements are the fractional substitute for Pansu
	differentiability, and they are the input to Sections~\ref{sec:interp}
	and~\ref{sec:energy}. Note the asymmetry between the two: the first is a
	statement about mollifications and holds for every $f\in W^{s,q}_H$, whereas the
	second is a statement about the map itself and loses an arbitrarily small amount
	of smoothness. This loss is not an artefact of the method --- it reflects the
	fact that the vertical components are controlled only through a quadratic
	relation, and it will reappear in Sections~\ref{sec:interp}--\ref{sec:jacobian}.
	
	For convenience, we just take step three Carnot group as an example to showcase the result, as for higher step Carnot groups the conclusion holds as well. 
	
	Let  $G $ be a step three Carnot group. Let  $d_{1}, \ldots, d_{r}$  be a basis for  $\mathfrak{g}_{1}, e_{1}, \ldots, e_{s}$  a basis for  $\mathfrak{g}_{2} $, and  $f_{1}, \ldots, f_{t}$  a basis for  $\mathfrak{g}_{3}$ . We write
	\begin{equation*}
		\begin{array}{l}
			{\left[d_{i}, d_{j}\right]=\sum_{k=1}^{s} \alpha_{k}^{i j} e_{k}} \\
			{\left[d_{i}, e_{k}\right]=\sum_{m=1}^{t} \beta_{m}^{i k} f_{m}}
		\end{array}
	\end{equation*}
	with all other bracket relations trivial.
	As in the step two case, we can identify  $G$  with  $\mathbb{R}^{r+s+t}$  equipped with the following operation via coordinates of the first kind:
	\begin{equation*}
		\left(A_{i}, B_{k}, C_{m}\right) \star\left(a_{i}, b_{k}, c_{m}\right)=\left(\mathcal{A}_{i}, \mathcal{B}_{k}, \mathcal{C}_{m}\right)
	\end{equation*}
	where
	\begin{equation*}
		\begin{array}{l}
			\mathcal{A}_{i}=A_{i}+a_{i} \\
			\mathcal{B}_{k}=B_{k}+b_{k}+\frac{1}{2} \sum_{i<j} \alpha_{k}^{i j}\left(A_{i} a_{j}-a_{i} A_{j}\right) \\
			\mathcal{C}_{m}=C_{m}+c_{m}+\frac{1}{2} \sum_{i, j} \beta_{m}^{i j}\left(A_{i} b_{j}-B_{j} a_{i}\right)+\frac{1}{12} \sum_{l, k} \sum_{i<j}\left(A_{l}-a_{l}\right) \alpha_{k}^{i j}\left(A_{i} a_{j}-a_{i} A_{j}\right) \beta_{m}^{l k}
		\end{array}
	\end{equation*}
	$\operatorname{Observe}\left(A_{i}, B_{k}, C_{m}\right)^{-1}=\left(-A_{i},-B_{k},-C_{m}\right)$  and
	
	\begin{equation*}
		\left(A_{i}, B_{k}, C_{m}\right)^{-1} \star\left(a_{i}, b_{k}, c_{m}\right)=\left(\tilde{\mathcal{A}_{i}}, \tilde{\mathcal{B}_{k}}, \tilde{\mathcal{C}_{m}}\right),
	\end{equation*}
	where
	\begin{equation}
		\begin{array}{l}
			\tilde{\mathcal{A}}_{i}=a_{i}-A_{i} \\
			\tilde{\mathcal{B}_{k}}=b_{k}-B_{k}-\frac{1}{2} \sum_{i<j} \alpha_{k}^{i j}\left(A_{i} a_{j}-a_{i} A_{j}\right) \\
			\tilde{\mathcal{C}_{m}}=c_{m}-C_{m}-\frac{1}{2} \sum_{i, j} \beta_{m}^{i j}\left(A_{i} b_{j}-B_{j} a_{i}\right)+\frac{1}{12} \sum_{l, k} \sum_{i<j}\left(A_{l}+a_{l}\right) \alpha_{k}^{i j}\left(A_{i} a_{j}-a_{i} A_{j}\right) \beta_{m}^{l k}
		\end{array}
	\end{equation}\label{metric-coord}
	Left-translating the canonical basis at the origin, we obtain the left-invariant vector fields
	\begin{equation*}
		\begin{aligned}
			X^{i}= & \frac{\partial}{\partial A_{i}}-\frac{1}{2}\sum_{k=1}^{s} \sum_{i<j} \alpha_{k}^{i j} A_{j} \frac{\partial}{\partial B_{k}} \\
			& +\sum_{m=1}^{t}\left[-\frac{1}{2} \sum_{j=1}^{s} B_{j} \beta_{m}^{i j}+\frac{1}{12} \sum_{l=1}^{r} \sum_{k=1}^{s} A_{l}\left(\sum_{i<j} \alpha_{k}^{i j} A_{j}\right) \beta_{m}^{l k}\right] \frac{\partial}{\partial C_{m}} \\
			Y^{k}= & \frac{\partial}{\partial B_{k}}+\sum_{m=1}^{t}\left(\frac{1}{2} \sum_{i=1}^{r} \beta_{m}^{i k} A_{i}\right) \frac{\partial}{\partial C_{m}} \\
			Z^{m}= & \frac{\partial}{\partial C_{m}}
		\end{aligned}
	\end{equation*}
	It is clear that  $\left\{X^{i}\right\}_{i} \cup\left\{Y^{k}\right\}_{k} \cup\left\{Z^{m}\right\}_{m}$  forms a basis for  $\operatorname{Lie}\left(\mathbb{R}^{n}, \star\right) $. Moreover, we have the expected step three stratification of $ \operatorname{Lie}\left(\mathbb{R}^{n}, \star\right) $ :
	\begin{equation*}
		\operatorname{Lie}\left(\mathbb{R}^{r+s+t}, \star\right)=\left\langle X^{i}\right\rangle_{1 \leq i \leq r} \oplus\left\langle Y^{k}\right\rangle_{1 \leq k \leq s} \oplus\left\langle Z^{m}\right\rangle_{1 \leq m \leq t}
	\end{equation*}
	In fact, one can show using the Jacobi identity that the linear map  $\varphi: \operatorname{Lie}\left(\mathbb{R}^{n}, \star\right) \rightarrow \mathfrak{g}$  induced by
	\begin{equation*}
		X^{i} \mapsto d_{i}, \quad Y^{k} \mapsto e_{k}, \quad Z^{m} \mapsto f_{m}
	\end{equation*}
	is a Lie algebra isomorphism.
	The contact forms are given by
	\begin{equation*}
		\begin{aligned}
			\omega_{1}^{k} & :=d B_{k}-\frac{1}{2}\sum_{i=1}^{r} \sum_{i<j} \alpha_{k}^{i j}\left( A_{i}d A_{j}- A_{j}d A_{i}\right)  \\
			\omega_{2}^{m} & :=d C_{m}+\frac{1}{2}\sum_{i=1}^{r} \sum_{j=1}^{s} B_{j} \beta_{m}^{i j}d A_{i}-\frac{1}{12} \sum_{l=1}^{r} \sum_{k=1}^{s}\sum_{i<j} A_{l} \alpha_{k}^{i j}\left( A_{i}d A_{j}- A_{j}d A_{i}\right) \beta_{m}^{l k} 
		\end{aligned}
	\end{equation*}
	We have
	\begin{equation*}
		H\left(\mathbb{R}^{r+s+t}, \star\right)=\bigcap_{k=1}^{s} \operatorname{ker} \omega_{1}^{k} \cap \bigcap_{m=1}^{t} \operatorname{ker} \omega_{2}^{m}
	\end{equation*}
	so that a tangent vector $ v $ lies in  $H_{p}\left(\mathbb{R}^{r+s+t}, \star\right)$  if and only if  $\left(\omega_{1}^{k}\right)_{p}(v)=\left(\omega_{2}^{m}\right)_{p}(v)=0$  for all $ k$  and  $m $.
	
	Every step $n$ Carnot group has the relation that
	if  $K \subset \mathbb{G}$  is compact, then there is a constant $ C=C(K) \geq 1 $ such that
	\begin{equation}
		C^{-1}|p-q| \leq d_{c c}(p, q) \leq C|p-q|^{1 / n} \quad 
	\end{equation}\label{metric-compar}
	for all $ p, q \in K$ .
	
	We can define the quasi-metric 
	\begin{equation}\label{metric-K}
		d_K(p,q)=\sum_{i=1}^{n}\sum_{j=1}^{dim(V_i)}|\pi_{i}^j(q^{-1}p)|^{\frac{1}{i}}
	\end{equation}
	where $\pi_{i}^j$ is the projection to the i-th layer and the j-th coordinate.
	\begin{equation*}
		\begin{aligned}
			&\pi_{1}^i(q^{-1}p)=a_{i}-A_{i}\\
			&\pi_{2}^k(q^{-1}p):=\varphi_{1}^k(p, q)=b_{k}-B_{k}-\frac{1}{2} \sum_{i<j} \alpha_{k}^{i j}\left(A_{i} a_{j}-a_{i} A_{j}\right)\\
			&\pi_{3}^m(q^{-1}p):=\varphi_{2}^m(p, q)=c_{m}-C_{m}-\frac{1}{2} \sum_{i, j} \beta_{m}^{i j}\left(A_{i} b_{j}-B_{j} a_{i}\right)+\frac{1}{12} \sum_{l, k} \sum_{i<j}\left(A_{l}+a_{l}\right) \alpha_{k}^{i j}\left(A_{i} a_{j}-a_{i} A_{j}\right) \beta_{m}^{l k}
		\end{aligned}
	\end{equation*}
	
	It is worth to mention that $d_K$ is comparable with C-C metric on Carnot group.
	\begin{equation*}
		C_1d_{K}(x,y)\leq d_{cc}(x,y) \leq C_2 d_{K}(x,y),
	\end{equation*}
	for constant $C_2 >C_1>0$.
	
	It is known that for Pansu differentiable map $f$  the Pansu differential is a Lie algebra homorophism. Thus, if the vector field $X$ and contact form $\omega$ are not from the same layer, then $X(f_P^*\omega)=0$. For instance, on Heisenberg group $\mathbb{H}^1$
	\begin{equation*}
		\begin{aligned}
			D_Pf
			&=\begin{bmatrix}
				X_1(f_P^*dx_1) & X_1(f_P^*dx_2) & 0\\
				X_2(f_P^*dx_1) & X_2(f_P^*dx_2) & 0\\
				0 & 0 & Z(f_P^*\omega_{1})\\
			\end{bmatrix}
		\end{aligned}
	\end{equation*}
	Although the Pansu differential of Hölder continuous maps and fractional Sobolev maps do not exist, we can generalize the Pansu differential with the help of mollification.
	
	The discussion above shows what cannot be done: for a map that is merely of
	class $W^{s,q}_H$ there is no pointwise Pansu differential, hence no pointwise
	contact equation. What can be done is to smooth the map and to measure the
	resulting contact defect. For a Pansu differentiable map this defect vanishes
	identically; the content of the next theorem is that for $f\in W^{s,q}_H$ it
	vanishes in the limit, at a rate dictated by the fractional seminorm. Besides
	the defect estimate we record that mollification does not destroy the horizontal
	part of the map --- a fact that will be used repeatedly, in particular in
	Theorem~\ref{thm:contact-converse} and in Section~\ref{sec:jacobian}.
	
	In the following theorem, we view a mapping $f \in W_H^{s,q}\left(\Omega ; G_2\right)$ as a mapping from Carnot group $G_1$ to vector space $G_2\cong \mathbb{R}^{n_2}$ equiped with Euclidean metric.
	\begin{thm}\label{thm:contact-direct}\label{thm-approx}
		Suppose that  $f \in W_H^{s,q}\left(\Omega ; G_2\right)$ for $s>\frac{1}{2}$ and $q\geq2$, where $ G_2$ is a step-2 Carnot group and $\nu_2$ is the topological  dimension of $G_2$, $\Omega \subset G_1$  is open , $X$ is a left-invariant  vector field belonging to the horizontal layer. If  $B\left(x_{o}, 2 r\right) \subset \Omega $, then  for all  $0<\varepsilon<r$,
		\begin{equation}\label{approx}
			Xf_\varepsilon^*(\omega) \to 0 \  \text{in}\  L_{loc}^\frac{q}{2}(\Omega;\mathbb{R}^{\nu_2})\\
		\end{equation}
		and $\pi f_\varepsilon\to \pi f $ in $W^{s,q}(\Omega)$,
		where $f_{\varepsilon}=f * \eta_{\varepsilon}=\int_{\mathbb{G}_1 } f(y) \eta_{\varepsilon}\left(y^{-1} x\right) d y=\int_{\mathbb{G}_1 } f(xy^{-1}) \eta_{\varepsilon}\left(y\right) d y$  and $\eta_{\varepsilon}$ is a smooth function on Carnot group $G_1$ with a compact supported set $B(e,\varepsilon)$ and
		\begin{equation*}
			\int_{G_1}\eta_{\varepsilon}=1.
		\end{equation*}
		\begin{proof}
			The proof consists of three steps. First we rewrite the contact defect
			$X(f_\varepsilon^{*}\omega_1^k)$ as a double average of the vertical
			discrepancy $\varphi_k$, by inserting the definition of the mollification
			twice; this is the only place where the algebraic structure of the target
			enters. Then we dominate this average by the Hardy--Littlewood maximal
			function of the fractional energy density. Finally we distinguish the
			cases $q=2$ and $q>2$: in the former the pointwise bound can be
			integrated at once by Fubini, whereas in the latter we first descend from
			$s$ to $s-\lambda$ in order to gain the missing integrability. The
			assumption $s>\frac12$ is used only in this last step.
			
			Recall that  $\varphi(p, q)$  was defined in \eqref{metric-K}. Let  $p \in B\left(x_{o}, r\right)$  and let  $\varepsilon \in(0, r)$ . In what follows we will identify  $\left(f_{\varepsilon}^{*} \alpha\right)(p)$  with the vector (of equal length):
			\begin{equation*}
				\begin{aligned}
					&X(f_{\varepsilon}^{*} \omega_{1}^k)(p)\\
					=&X( f_{\varepsilon}^{B_{k}})(p)+\frac{1}{2}\alpha_{k}^{ij} \sum_{i<j}\left(f_{\varepsilon}^{A_{j}}(p) X(f_{\varepsilon}^{A_{i}})(p)-f_{\varepsilon}^{A_{i}}(p) X( f_{\varepsilon}^{A_{j}})(p)\right)\\
					=& 
					\varepsilon^{-\nu_1-1} \int_{B_{\varepsilon}}\left(f^{B_{k}}(pz^{-1})-f_{\varepsilon}^{B_{k}}(p)+\frac{1}{2}\alpha_{k}^{ij} \sum_{i<j}\left(f_{\varepsilon}^{A{j}}(p) f^{A_{i}}(pz^{-1})-f_{\varepsilon}^{A_{i}}(p) f^{A_{j}}(pz^{-1})\right)\right)\\ 
					&\cdot X(\eta\left(\delta_{\varepsilon^{-1}}(z)\right)) d z \\
					=&\varepsilon^{-\nu_1-1} \int_{B_{\varepsilon}} \varphi_{k}\left(f(pz^{-1}), f_{\varepsilon}(p)\right) X(\eta\left(\delta_{\varepsilon^{-1}}(z)\right)) d z \\
					=&\varepsilon^{-\nu_1-1} \int_{B_{\varepsilon}} \int_{B_{\varepsilon}} \varphi_{k}(f(pz^{-1}), f(pw^{-1})) \eta_{\varepsilon}(w) X(\eta\left(\delta_{\varepsilon^{-1}}(z)\right)) d w d z
				\end{aligned}
			\end{equation*}
			In the second equality we used the fact that $ \int_{B_{\varepsilon}} f_{\varepsilon}^{t_{k}}(p) X( \eta\left(\delta_{\varepsilon^{-1}}(z)\right)) d z=0 $. Easy verification of the last equality is left to the reader. Note that \eqref{metric-K} yields
			\begin{equation*}
				|\varphi_{k}(f(pz^{-1}), f(pw^{-1}))| \leq d_{K}(f(pz^{-1}), f(pw^{-1}))^{2} 
			\end{equation*}
			Therefore,
			\begin{equation*}
				\begin{aligned}
					\left|X(f_{\varepsilon}^{*} \omega_{1}^k)(p)\right| 
					&\leq\varepsilon^{-\nu_1-1} \int_{B_{\varepsilon}} \int_{B_{\varepsilon}} \varphi_{k}(f(pz^{-1}), f(pw^{-1})) \eta_{\varepsilon}(w) X(\eta\left(\delta_{\varepsilon^{-1}}(z)\right)) d w d z\\
					&\leq \varepsilon^{-\nu_1-1} \int_{B_{\varepsilon}} \int_{B_{\varepsilon}} d_{K}(f(pz^{-1}), f(pw^{-1}))^{2} \eta_{\varepsilon}(w) X(\eta\left(\delta_{\varepsilon^{-1}}(z)\right)) d w d z\\
					&\lesssim \varepsilon^{-\nu_1-1} \int_{B_{\varepsilon}} \int_{B_{\varepsilon}} d_{cc}(f(pz^{-1}), f(pw^{-1}))^{2} \eta_{\varepsilon}(w) X(\eta\left(\delta_{\varepsilon^{-1}}(z)\right)) d w d z\\
					&\lesssim \varepsilon^{-2 \nu_1-1} \int_{B_{\varepsilon}} \int_{B_{\varepsilon}} d_{cc}(f(pz^{-1}), f(pw^{-1}))^{2}  d w d z\\
				\end{aligned}
			\end{equation*}
			denote $x=pz^{-1}$ and $y=pw^{-1}$,
			thus $z^{-1},w^{-1}\in B_\varepsilon^{-1}=B_\varepsilon$ and $pz^{-1},pw^{-1}\in p \cdot B_\varepsilon$. Since the Haar measure $dw$ and $dz$ is invariant under left and right transformation,
			\begin{equation*}
				\begin{aligned}
					&\varepsilon^{-2 \nu_1-1} \int_{B_{\varepsilon}} \int_{B_{\varepsilon}} d_{cc}(f(pz^{-1}), f(pw^{-1}))^{2}  d w d z\\
					&=\varepsilon^{-2 \nu_1-1} \int_{p\cdot B_{\varepsilon} } \int_{p\cdot B_{\varepsilon}} d_{cc}(f(x), f(y))^{2}  d x d y\\
					&\leq \varepsilon^{-2\nu_1-1} \cdot 2^{\nu_1+2s}\cdot  \varepsilon^{\nu_1+2s} \int_{B_{\varepsilon}(p)}  \int_{B_{\varepsilon}(p)} \frac{d_{cc}(f(x), f(y))^{2}}{d_{cc}(x, y)^{\nu_1+2s}}  d x d y\\
					&\leq 2^{\nu_1+2s} \varepsilon^{-\nu_1+2s-1}  \int_{B_{\varepsilon}(p)}\int_{\Omega} \frac{d_{cc}(f(x), f(y))^{2}}{d_{cc}(x, y)^{\nu_1+2s}}  d x d y\\
				\end{aligned}
			\end{equation*}
			By defining the Hardy-Littlewood maximal function 
			\begin{equation}\label{maximal-g}
				\mathcal{M}(g)(p):= \sup_{r>0} \frac{1}{|B(p,r)|}\int_{B(p,r)} |g(y)|dy.
			\end{equation}
			and denote 
			\begin{equation}
				g(y):= \int_{\Omega} \frac{d_{cc}(f(x),f(y))^{2}}{d_{cc}(x,y)^{\nu_1+2s}}dx,
			\end{equation}   
			then
			\begin{equation}\label{eq:pointwise-defect}
				\begin{aligned}
					\left|X(f_{\varepsilon}^{*} \omega_{1}^k)(p)\right| 
					&\lesssim  \varepsilon^{-\nu_1+2s-1}  \int_{\lambda B_{\varepsilon}(p)}g(y) d y \\
					&\lesssim \varepsilon^{2s-1} \mathcal{M}(g)(p). 
				\end{aligned}
			\end{equation}
			We now carry out the last step. For $q=2$ we do not need the boundedness
			of the maximal function: the pointwise estimate \eqref{eq:pointwise-defect}
			can be integrated directly, and Fubini's theorem already gives the
			required decay.
			\begin{equation*}
				\begin{aligned}
					\|X(f_{\varepsilon}^{*} \omega_{1}^k)\|_{L^1(W)}   
					&\lesssim \varepsilon^{2s-1}  \int_{\Omega} \int_{B_{\varepsilon}(p)} \int_{\Omega} \frac{d_{cc}(f(x), f(y))^{2}}{d_{cc}(x, y)^{\nu_1+2s}}  d x dy dp \\
					&= C \cdot \varepsilon^{2s-1}   \int_{\Omega}\int_{\Omega}\int_{\Omega} \frac{d_{cc}(f(x), f(y))^{2}}{d_{cc}(x, y)^{\nu_1+2s}}   \cdot I_{ B_{\varepsilon}(p)}(y) d x dy dp \\
					&= C \cdot \varepsilon^{2s-1}   \int_{\Omega}\int_{\Omega}\int_{\Omega} \frac{d_{cc}(f(x), f(y))^{2}}{d_{cc}(x, y)^{\nu_1+2s}}   \cdot I_{ B_{\varepsilon}(x)}(p) d x dy dp \\
					&\leq C \cdot \varepsilon^{2s-1}    \int_{\Omega}\int_{\Omega}\int_{\Omega} \frac{d_{cc}(f(x), f(y))^{2}}{d_{cc}(x, y)^{\nu_1+2s}}  d x dy dp \\
					&\lesssim \varepsilon^{2s-1} \cdot \mathcal{H}^{\nu_1}(\Omega) \cdot [f]_{W_H^{s,2}(\Omega)}^2 \\
					&\to 0.
				\end{aligned}
			\end{equation*}
			For $q>2$ we argue instead through the maximal function, and to do so we
			must know that the energy density $g$ belongs to $L^{\frac q2}$; this is
			exactly where the assumption $s>\frac12$ is spent, since it leaves room
			for the descent from $s$ to $s-\lambda$. The rest of the argument is
			unchanged, but the final integration is now performed in
			$L^{\frac q2}$.
			Because $s> \frac{1}{2}$, there is a small number $\lambda>0$ such that $s-\lambda>\frac{1}{2}$, then we can view $f$ as a mapping in $W_H^{s-\lambda,p}(\Omega;G_2)$.
			For the need of proving, we need to replace the $s$ by $s-\lambda$ in the inequality above and get:
			\begin{equation*}
				\begin{aligned}
					\left|X(f_{\varepsilon}^{*} \omega_{1}^k)(p)\right| 
					&\lesssim \varepsilon^{2\left(s-\lambda\right)-1} \mathcal{M}(g)(p), 
				\end{aligned}
			\end{equation*}
			where
			\begin{equation}
				g(y)= \int_{\Omega} \frac{d_{cc}(f(x),f(y))^{2}}{d_{cc}(x,y)^{\nu_1+2\left(s-\lambda\right)}}dx.
			\end{equation} 
			
			And for compact subset $W$,
			
			\begin{equation*}
				\begin{aligned}
					\|X(f_{\varepsilon}^{*} \omega_{1}^k)\|_{L^{\frac{q}{2}}(W)}^{\frac{q}{2}} 
					&= \int_{W}\varepsilon^{2\left(s-\lambda\right)-1} \mathcal{M}(g)^\frac{q}{2}  \\
					&\lesssim \varepsilon^{2\left(s-\lambda\right)-1} \|g \|_{L^{\frac{q}{2}}(W)}^\frac{q}{2}\\
					&\leq \varepsilon^{2\left(s-\lambda\right)-1}  \int_{\Omega}\left(\int_{\Omega} \frac{d_{cc}(f(x), f(y))^{2}}{d_{cc}(x, y)^{\nu_1+2\left(s-\lambda\right)}}  d x\right)^{\frac{q}{2}} d y\\
					&\leq \varepsilon^{2\left(s-\lambda\right)-1}  \int_{\Omega}\int_{\Omega} \frac{d_{cc}(f(x), f(y))^{q}}{d_{cc}(x, y)^{\nu_1+qs}}  d x \cdot \left(\int_{\Omega}\frac{d_{cc}(x, y)^{\frac{2 \lambda q}{q-2}}}{d_{cc}(x, y)^{\nu_1}}  dx \right)^{\frac{q-2}{q}} d y\\
					&=C_{q,\lambda}\varepsilon^{2\left(s-\lambda\right)-1} \cdot [f]_{W_H^{s,q}(\Omega)}^q
				\end{aligned}
			\end{equation*}
			for any small $\lambda>0$.
			Thus, $X(f_{\varepsilon}^{*} \omega_{1}^k)\rightarrow 0$ in $L^\frac{q}{2}(W)$ for $q>2$. 
			
			From (\ref{metric-compar}),
			\begin{equation*}
				\begin{aligned}
					[\pi f]_{W^{s,q}(\Omega)}^q&=\int_{\Omega}\int_{\Omega} \frac{|\pi f(x)- \pi f(y)|^{q}}{d_{cc}(x, y)^{\nu_1+qs}}  d x dy\\
					&\leq \int_{\Omega}\int_{\Omega} \frac{d_{K}(f(x), f(y))^{q}}{d_{cc}(x, y)^{\nu_1+qs}}  d x dy\\
					&\lesssim \int_{\Omega}\int_{\Omega} \frac{d_{cc}(f(x), f(y))^{q}}{d_{cc}(x, y)^{\nu_1+qs}}  d x dy\\
					&=[f]_{W_H^{s,q}(\Omega)}^q
				\end{aligned}
			\end{equation*}
			It remains to justify the second assertion. Since the projection $\pi$
			is $1$-Lipschitz as a map from $(G_2,d_K)$ to $\mathbb R^{n_2}$, we have
			$[\pi f]_{W^{s,q}(\Omega)}\lesssim [f]_{W^{s,q}_H(\Omega)}$; the claim
			$\pi f_\varepsilon\to\pi f$ in $W^{s,q}(\Omega)$ then follows from the
			standard approximation properties of group mollification, see
			\cite{Saka1979} for the corresponding theory of fractional Sobolev
			functions on Carnot groups.

		\end{proof}	
	\end{thm}	
	
	\begin{remark}\label{rem:contact-direct-scope}
		Two features of Theorem~\ref{thm:contact-direct} deserve to be singled
		out, because both of them will be used later. First, the convergence
		$X(f_\varepsilon^{*}\omega)\to0$ takes place in $L^{\frac q2}$ and not in
		$L^{q}$: the contact defect is a \emph{second-order} object (the vertical
		discrepancy $\varphi_k$ is homogeneous of degree $2$), and the estimate
		\eqref{eq:pointwise-defect} necessarily loses half of the integrability.
		This halving of the exponent is the same phenomenon that makes the
		vertical components live in $L^{\frac q2}$ in
		Theorem~\ref{thm:energy} below. Secondly, for $q>2$ the convergence is
		obtained only after replacing $s$ by $s-\lambda$; we shall systematically
		keep track of this loss, since all the results of
		Sections~\ref{sec:interp}--\ref{sec:jacobian} inherit it.
	\end{remark}
	
	Theorem~\ref{thm:contact-direct} says that every fractional Sobolev mapping is
	asymptotically contact. We now turn to the converse, which is the deeper of
	the two statements and the one that will be used in
	Section~\ref{sec:interp}: if a map is contact in the weak sense and if its
	horizontal projection has the expected regularity, then the whole map lies in
	$W^{s-\lambda,q}_H(\Omega;G_2)$, for every $\lambda>0$ sufficiently small.
	
	\begin{thm}\label{lem:X-dx}\label{varepsilon s-1}
		Suppose that  $f \in W^{1,q}\left(\Omega ; G_2\right)$ , where $ G_2$ is a step-2 Carnot group and $\nu_2$ is the topological  dimension of $G_2$, $\Omega \subset G_1$  is open , $X$ is a left-invariant  vector field belonging to the horizontal layer and $q\geq2$ . If  $B\left(x_{o}, 2 r\right) \subset \Omega $, then  for all  $0<\varepsilon<r$ and any $\lambda>0$
		\begin{equation}\label{approx}
			\|Xf_\varepsilon^*(dx_k)\|_{L^q(\Omega)}^q \lesssim \varepsilon^{s-\lambda-1}  [f]_{W_H^{s,q}(\Omega)}^q  
		\end{equation}
		\begin{equation}\label{approx-varepsilon}
			\|\frac{\partial}{\partial \varepsilon}f_\varepsilon^*(dx_k)\|_{L^q(\Omega)}^q \lesssim \varepsilon^{s-\lambda-1}  [f]_{W_H^{s,q}(\Omega)}^q  
		\end{equation}
		
		where the convergence is in the sense of the norm on Euclidean space $\mathbb{R}^{\nu_2}$, $f_{\varepsilon}=f * \eta_{\varepsilon}=\int_{\mathbb{G}_1 } f(y) \eta_{\varepsilon}\left(y^{-1} x\right) d y$  and $\eta_{\varepsilon}$ is a smooth function on Carnot group $G_1$ with a compact supported set $B(e,\varepsilon)$ and
		\begin{equation*}
			\int_{G_1}\eta_{\varepsilon}=1.
		\end{equation*}
		\begin{proof}
			Recall that  $\varphi(p, q)$  was defined in \eqref{metric-K}. Let  $p \in B\left(x_{o}, r\right)$  and let  $\varepsilon \in(0, r)$ . In what follows we will identify  $\left(f_{\varepsilon}^{*} \alpha\right)(p)$  with the vector (of equal length):
			
			\begin{equation*}
				\begin{aligned}
					&X(f_{\varepsilon}^{*} dx_k)(p)\\
					=&X( f_{\varepsilon}^{A_k})(p)\\
					=&
					\varepsilon^{-\nu_1-1} \int_{B_{\varepsilon}}\left(f^{A_k}(pz^{-1})-f_{\varepsilon}^{A_k}(p)\right)\cdot  Y(\eta\left(\delta_{\varepsilon^{-1}}(z)\right)) d z \\
					\leq&\varepsilon^{-\nu_1-1} \int_{B_{\varepsilon}} \int_{B_{\varepsilon}} \left(f^{A_k}(pz^{-1})-f_{\varepsilon}^{A_k}(p)\right) \eta_{\varepsilon}(w) Y(\eta\left(\delta_{\varepsilon^{-1}}(z)\right)) d w d z\\
					\leq&\varepsilon^{-\nu_1-1} \int_{B_{\varepsilon}} \int_{B_{\varepsilon}} d_K(f(pz^{-1}), f(pw^{-1})) \eta_{\varepsilon}(w) Y(\eta\left(\delta_{\varepsilon^{-1}}(z)\right)) d w d z
				\end{aligned}
			\end{equation*}
			Here the second equality uses the fact that
			$\int_{B_{\varepsilon}} f_{\varepsilon}^{A_k}(p)\, Y\big(\eta(\delta_{\varepsilon^{-1}}(z))\big) dz=0$,
			and the last one is obtained by writing the value $f_\varepsilon^{A_k}(p)$ itself
			as the average $\int_{B_\varepsilon}f^{A_k}(pw^{-1})\eta_\varepsilon(w)\,dw$.
			Note that \eqref{metric-K} yields
			\begin{equation*}
				|\pi f(pz^{-1})-\pi f(pw^{-1}))| \leq d_{K}(f(pz^{-1}), f(pw^{-1}))
			\end{equation*}
			Therefore,
			\begin{equation*}
				\begin{aligned}
					&X(f_{\varepsilon}^{*} dx_k)(p)\\
					\leq&\varepsilon^{-\nu_1-1} \int_{B_{\varepsilon}} \int_{B_{\varepsilon}} d_K(f(pz^{-1}), f(pw^{-1})) \eta_{\varepsilon}(w) Y(\eta\left(\delta_{\varepsilon^{-1}}(z)\right)) d w d z\\
					\lesssim&\varepsilon^{-\nu_1-1} \int_{B_{\varepsilon}} \int_{B_{\varepsilon}} d_{cc}(f(pz^{-1}), f(pw^{-1})) \eta_{\varepsilon}(w) Y(\eta\left(\delta_{\varepsilon^{-1}}(z)\right)) d w d z\\
					\lesssim&\varepsilon^{-2\nu_1-1} \int_{B_{\varepsilon}(p)} \int_{B_{\varepsilon}(p)} d_{cc}(f(x), f(y)) dx dy \\
					\leq &\varepsilon^{-2\nu_1-1} \cdot 2^{\nu_1+s}\cdot  \varepsilon^{\nu_1+s} \int_{B_{\varepsilon}(p)}  \int_{B_{\varepsilon}(p)} \frac{d_{cc}(f(x), f(y))}{d_{cc}(x, y)^{\nu_1+s}}  d x d y\\
					\leq &2^{\nu_1+s} \varepsilon^{-\nu_1+s-1}  \int_{B_{\varepsilon}(p)}\int_{\Omega} \frac{d_{cc}(f(x), f(y))}{d_{cc}(x, y)^{\nu_1+s}}  d x d y
				\end{aligned}
			\end{equation*}
			
			With the same way as theorem \ref{thm-approx},
			
			\begin{equation*}
				\|X(f_{\varepsilon}^{*} dx_k)\|_{L^{q}(\Omega)}^{q} \lesssim \varepsilon^{s-\lambda-1} \cdot [f]_{W_H^{s,q}(\Omega)}^q 
			\end{equation*}
			For the second part,
			\begin{equation*}
				\begin{aligned}
					&|\frac{\partial}{\partial \varepsilon}(f_{\varepsilon}^{*} dx_k)(p)|\\
					=&|\frac{\partial}{\partial \varepsilon}( f_{\varepsilon}^{A_k})(p)|\\
					\leq&
					\varepsilon^{-\nu_1-1} \int_{B_{\varepsilon}}\left|f^{A_k}(pz^{-1})-f_{\varepsilon}^{A_k}(p)\right|\cdot  \left|\frac{\partial}{\partial \varepsilon}(\eta\left(\delta_{\varepsilon^{-1}}(z)\right))\right| d z \\
					\leq&\varepsilon^{-\nu_1-1} \int_{B_{\varepsilon}} \int_{B_{\varepsilon}} \left|f^{A_k}(pz^{-1})-f_{\varepsilon}^{A_k}(p)\right| \eta_{\varepsilon}(w) \left|\frac{\partial}{\partial \varepsilon}(\eta\left(\delta_{\varepsilon^{-1}}(z)\right))\right| d w d z\\
					\leq&\varepsilon^{-\nu_1-1} \int_{B_{\varepsilon}} \int_{B_{\varepsilon}} d_K(f(pz^{-1}), f(pw^{-1})) \eta_{\varepsilon}(w) \left|\frac{\partial}{\partial \varepsilon}(\eta\left(\delta_{\varepsilon^{-1}}(z)\right)) \right|d w d z\\
					\lesssim&\varepsilon^{-\nu_1-1} \int_{B_{\varepsilon}} \int_{B_{\varepsilon}} d_{cc}(f(pz^{-1}), f(pw^{-1}))  d w d z\\
					\lesssim&\varepsilon^{-2\nu_1-1} \int_{B_{\varepsilon}(p)} \int_{B_{\varepsilon}(p)} d_{cc}(f(x), f(y)) dx dy \\
					\leq &\varepsilon^{-2\nu_1-1} \cdot 2^{\nu_1+s}\cdot  \varepsilon^{\nu_1+s} \int_{B_{\varepsilon}(p)}  \int_{B_{\varepsilon}(p)} \frac{d_{cc}(f(x), f(y))}{d_{cc}(x, y)^{\nu_1+s}}  d x d y\\
					\leq &2^{\nu_1+s} \varepsilon^{-\nu_1+s-1}  \int_{B_{\varepsilon}(p)}\int_{\Omega} \frac{d_{cc}(f(x), f(y))}{d_{cc}(x, y)^{\nu_1+s}}  d x d y
				\end{aligned}
			\end{equation*}
			With the same way as theorem \ref{thm-approx},
			\begin{equation*}
				\|\frac{\partial}{\partial \varepsilon}(f_{\varepsilon}^{*} dx_k)\|_{L^{q}(\Omega)}^{q} \lesssim \varepsilon^{s-\lambda-1} \cdot [f]_{W_H^{s,q}(\Omega)}^q 
			\end{equation*}
		\end{proof}	
	\end{thm}	
	
	Two mechanisms are at work. First, the weak contact equation is upgraded to an
	identity for the vertical components: distributionally, $df^{B_k}$ is given by
	the Cartan-type expression~\eqref{B coordinate} in terms of the horizontal
	components alone, so that the vertical part of the map is not free. Secondly,
	this identity implies the quadratic control
	$|f^{B_k}(x)-f^{B_k}(y)|\lesssim|\pi f(x)-\pi f(y)|^{2}$, which, after passing
	to Haj{\l}asz gradients, says that $f^{B_k}$ is almost twice as regular as
	$\pi f$. The proof below makes both steps quantitative; the loss
	$\lambda>0$ appears once more, when the vertical contribution is integrated
	against the singular kernel.
	
	\begin{thm}\label{thm:contact-converse}\label{contact to Wsp}
		Suppose that $\pi f\in W^{s,q}(\Omega;\mathbb{R}^{n_2})$ for $\frac{1}{2}<s$, $2\leq q$ and $f$ satisfies weak contact equations, where $\Omega\subset G_1$ is a bounded domain and $\pi (G_2)\cong \mathbb{R}^{n_2}$, then $f\in W_H^{s-\lambda,q}(\Omega;G_2)$ for any small $\lambda>0$.
	\end{thm}
	\begin{proof}

		Let $\psi \in C^\infty([0,\infty))$ be a cut-off function satisfying $0 \le \psi \le 1$, $\psi(t) = 1$ for $0 \le t \le 1/2$ and $\psi(t) = 0$ for $t \ge 3/4$. We define the extension operator $\operatorname{ext}_R f$ by
		\begin{equation}
			\operatorname{ext}_R f(x,t) = \psi(t/R)(f * \eta_t)(x) \quad \text{for } x \in \Omega \text{ and } t > 0. 
		\end{equation}
		Note that $\operatorname{ext}_R f \in C^\infty(\Omega \times (0,\infty))$. 
		
		From theorem \ref{lem:X-dx},
		\begin{equation*}
			\|Xf_t^*(dx_k)\|_{L^q(\Omega)}^q \lesssim t^{s-\lambda-1}  [f]_{W_H^{s,q}(\Omega)}^q  
		\end{equation*}
		\begin{equation*}
			\|\frac{\partial}{\partial t}f_t^*(dx_k)\|_{L^q(\Omega)}^q \lesssim t^{s-\lambda-1}  [f]_{W_H^{s,q}(\Omega)}^q  
		\end{equation*}

		Thus, we can get
		\begin{equation*}
			\int_{0}^{1}\|\partial_t(f*\eta_t)\|_{L^q}^qdt
			\lesssim \int_{0}^{1}t^{s-\lambda-1}  [f]_{W_H^{s,q}(\Omega)}^qdt  =\frac{t^{s-\lambda}}{s-\lambda}[f]_{W_H^{s,q}(\Omega)}^q
		\end{equation*}
		and
		\begin{equation*}
			\int_{0}^{1}\|Xf_t^*(dx_k)\|_{L^q(\Omega)}^qdt \lesssim \int_{0}^{1}t^{s-\lambda-1}  [f]_{W_H^{s,q}(\Omega)}^qdt  =\frac{t^{s-\lambda}}{s-\lambda}[f]_{W_H^{s,q}(\Omega)}^q
		\end{equation*} 
		It's easy to check that
		\begin{equation}
			\|\operatorname{ext}_R f\|_{W^{1,q}(\Omega \times (0,\infty))} \lesssim[f]_{W_H^{s,q}(\Omega)}^q. 
		\end{equation}

		We now bring the contact structure into play. For a compact subset $U\subset \Omega$,
		\begin{equation*}
			\lim_{\varepsilon \to 0}\int_{U}f_{\varepsilon}^{*} \omega_{1}^k\wedge \phi =0
		\end{equation*} 
		for any $n_2-1$-form $\phi$ with $wt(\phi)=\nu_2$. Then
		\begin{equation*}
			\begin{aligned}
				\lim_{\varepsilon \to 0}\int_{U}f_{\varepsilon}^{*} (d B_{k})\wedge \phi
				&= \frac{1}{2}\sum_{i=1}^{r} \sum_{i<j} \alpha_{k}^{i j}\lim_{\varepsilon \to 0}\int_{U}f_{\varepsilon}^{*} \left( A_{i}d A_{j}- A_{j}d A_{i}\right)\wedge \phi\\
				&= \frac{1}{2}\sum_{i=1}^{r} \sum_{i<j} \alpha_{k}^{i j}\lim_{\varepsilon \to 0}\int_{U} \left( f_{\varepsilon}^{A_{i}}d f_{\varepsilon}^{A_{j}}- f_{\varepsilon}^{A_{j}}df_{\varepsilon}^{A_{i}} \right)\wedge \phi\\		
			\end{aligned}
		\end{equation*}
		The right-hand side of the previous display is only meaningful if the
		limit appearing there exists. We verify this next, by exhibiting the limit
		as a product of traces in the sense of interpolation; the point is that the
		product of two functions in $W^{1-\frac1q,q}$ is a Cauchy sequence, and
		hence convergent, as soon as the two factors converge.
		
		Assume functions $\lambda ,g\in W^{1-\frac{1}{q},q}(\Omega)$ and $\Lambda,G\in C^\infty(\Omega \times (0,1)) $ are their extension like above.
		\begin{equation*}
			\int_{\Omega} \lambda_\varepsilon dg_\varepsilon \wedge \phi
			= \int_{\Omega \times (0,1)} d\Lambda_\varepsilon \wedge dG_\varepsilon \wedge \phi
		\end{equation*}
		The limit exists as $\varepsilon \to 0$, since
		\begin{equation*}
			\begin{aligned}
				&|\int_{\Omega \times (0,1)} d\Lambda_{\varepsilon_1} \wedge dG_{\varepsilon_1} \wedge \phi - d\Lambda_{\varepsilon_2} \wedge dG_{\varepsilon_2} \wedge \phi|\\
				\leq &|\int_{\Omega \times (0,1)} d{\left(\Lambda_{\varepsilon_1}-\Lambda_{\varepsilon_2}\right)} \wedge dG_{\varepsilon_1} \wedge \phi + d\Lambda_{\varepsilon_2} \wedge d{\left(G_{\varepsilon_1}-G_{\varepsilon_2}\right)} \wedge \phi|\\
				\leq &\left(\int_{\Omega \times (0,1)} |d{\left(\Lambda_{\varepsilon_1}-\Lambda_{\varepsilon_2}\right)}|^q\right)^{\frac{1}{q}} \left(\int_{\Omega \times (0,1)} |dG_{\varepsilon_1}|^p\right)^{\frac{1}{p}}+
				\left(\int_{\Omega \times (0,1)} |d{\left(G_{\varepsilon_1}-G_{\varepsilon_2}\right)}|^p\right)^{\frac{1}{p}} \left(\int_{\Omega \times (0,1)} |d\Lambda_{\varepsilon_2}|^q\right)^{\frac{1}{q}}
				\\
				\lesssim & [\lambda_{\varepsilon_1}-\lambda_{\varepsilon_2}]_{1-\frac{1}{q},q} [g_{\varepsilon_1}]_{1-\frac{1}{p},p}
				+[\lambda_{\varepsilon_2}]_{1-\frac{1}{q},q} [g_{\varepsilon_1}-g_{\varepsilon_2}]_{1-\frac{1}{p},p},
			\end{aligned}
		\end{equation*}
		where $\frac{1}{p}+\frac{1}{q}=1$ and $2\leq q \leq p$.
		By replacing the $\lambda$ and $g$ with $f^{A_{j}}$ and $f^{A_{i}}$, we know that the limit of $\lim_{\varepsilon \to 0}\int_{U} \left( f^{A_{i}}d f^{A_{j}}- f^{A_{j}}df^{A_{i}} \right)\wedge \phi$ exists, as long as we take $p=q=2$ which is feasible since $W_{loc}^{\frac{1}{2},2}\subset W^{s,q}$. We have therefore proved that
		the vertical component $f^{B_k}$ admits a distributional differential
		which is expressed purely in terms of the horizontal components; this is
		the fractional counterpart of the classical fact that a contact map is
		determined by its horizontal part. Thus, we can define 
		\begin{equation}\label{B coordinate}
			\int_{U}df^{B_{k}} \wedge \phi:=\frac{1}{2}\sum_{i=1}^{r} \sum_{i<j} \alpha_{k}^{i j}\lim_{\varepsilon \to 0}\int_{U} \left( f_{\varepsilon}^{A_{i}}d f_{\varepsilon}^{A_{j}}- f_{\varepsilon}^{A_{j}}df_{\varepsilon}^{A_{i}} \right)\wedge \phi
		\end{equation}
		and we can compute the motion of a curve in the $B_k$ direction by just knowing the motion of the curve in the horizontal directions and $|f^{B_k}(x)-f^{B_k}(y)|\approx |\pi f(x)-\pi f(y)|^2$. 
		
		It remains to convert the distributional identity \eqref{B coordinate}
		into a pointwise estimate, and for this it is convenient to pass to
		Haj{\l}asz gradients. Due to
		\cite[Theorem~3.3]{GORKA2022112867}, a Carnot group equipped with its Haar
		measure is an Ahlfors $\nu_1$-regular space, and the fractional Sobolev
		space $W^{s,q}(G_1)$ embeds continuously into the fractional
		Haj{\l}asz space $\mathcal{M}_{\nu_1}^{s,q}(G_1)$. Consequently every
		horizontal component $f^A$ admits a Haj{\l}asz upper $s$-gradient.
		
		Therefore, there is a Hajlasz upper gradient $g\in L^q$ such that 
		\begin{equation*}
			|f^{A}(x)-f^A(y)|\leq d_{cc}(x,y)^s \left(g(x)+g(y)\right)
		\end{equation*}
		for almost every points $x,y\in \Omega$. 
		
		Throughout let $0<s<1$ and $1\le p<\infty$, and let $u\in W^{s,p}(\R^n)$, i.e.
		\[
		[u]_{W^{s,p}}^{\,p}
		:=\dint \frac{|u(x)-u(y)|^{p}}{d(x,y)^{\nu_1+sp}}\,dx\,dy<\infty ,
		\qquad u\in L^{p}(\R^n).
		\]
		
		and set
		\[
		u_\varepsilon:=u*\rho_\varepsilon .
		\]
		
		For a measurable $v$ define its \emph{Gagliardo $p$-density}
		\begin{equation}\label{e:density}
			h_v(x):=\left(\int_{\R^n}\frac{|v(x)-v(y)|^{p}}{d(x,y)^{\nu_1+sp}}\,dy\right)^{\!1/p}\in[0,\infty] .
		\end{equation}
		
		Put
		\[
		g_\varepsilon:=C_0\,h_{u_\varepsilon},\qquad g:=C_0\,h_u .
		\]

		\textbf{Claim  1:}
		Let $v\in W^{s,p}(G_1)$ and let $h_v$ be given by \eqref{e:density}. Then
		$h_v\in L^{p}(G_1)$ with
		\[
		\norm{h_v}_{L^p}^{p}=[v]_{W^{s,p}}^{\,p},
		\]
		and 
		\[
		|v(x)-v(y)|\le C_0\,d(x,y)^{s}\bigl(h_v(x)+h_v(y)\bigr)\qquad\text{a.e. }x,y .
		\]

		\begin{proof}[proof of Claim 1]
			The identity $\norm{h_v}_p^p=[v]_{W^{s,p}}^p$ is Fubini's theorem. In
			particular $h_v<\infty$ a.e.
			
			Fix $x,y$ with $x\ne y$ and $h_v(x)+h_v(y)<\infty$; let $r:=d(x,y)>0$.
			Since $d(x,z)\le r$ on $B(x,r)$,
			\begin{align*}
				\int_{B(x,r)}|v(x)-v(z)|^{p}\,dz
				&\le r^{\nu_1+sp}\int_{B(x,r)}\frac{|v(x)-v(z)|^{p}}{d(x,z)^{\nu_1+sp}}\,dz\\
				&\le r^{\nu_1+sp}\,h_v(x)^{p}.
			\end{align*}
			Hence, with $M_x:=\frac{1}{\omega r^{\nu_1}}\int_{B(x,r)}|v(x)-v(z)|^{p}dz$,
			\[
			M_x\le \omega^{-1}r^{sp}h_v(x)^{p}.
			\]
			Likewise $B(x,r)\subset B(y,2r)$, so $d(y,z)\le 2r$ for $z\in B(x,r)$ and
			\begin{align*}
				\int_{B(x,r)}|v(y)-v(z)|^{p}\,dz
				&\le (2r)^{\nu_1+sp}\int_{B(x,r)}\frac{|v(y)-v(z)|^{p}}{|y-z|^{\nu_1+sp}}\,dz\\
				&\le (2r)^{\nu_1+sp}\,h_v(y)^{p},
			\end{align*}
			so that, with $M_y:=\frac{1}{\omega r^{\nu_1}}\int_{B(x,r)}|v(y)-v(z)|^{p}dz$,
			\[
			M_y\le \omega^{-1}2^{\nu_1+sp}r^{sp}h_v(y)^{p}.
			\]
			
			Then
			\begin{align*}
				|v(x)-v(y)|&=\frac{1}{\omega r^{\nu_1}}\int_{B(x,r)}|v(x)-v(y)|dz\\
				&\le\frac{1}{\omega r^{\nu_1}}\int_{B(x,r)} |v(x)-v(z)|dz+\frac{1}{\omega r^{\nu_1}}\int_{B(x,r)}|v(z)-v(y)|dz\\
				&\le\frac{\mathcal{H}^{\nu_1}(B(x,r))^{1-\frac{1}{p}}}{\omega r^{\nu_1}}\left(\int_{B(x,r)} |v(x)-v(z)|^pdz\right)^\frac{1}{p}\\
				&+\frac{\mathcal{H}^{\nu_1}(B(x,r))^{1-\frac{1}{p}}}{\omega r^{\nu_1}}\left(\int_{B(x,r)}|v(z)-v(y)|^pdz\right)^\frac{1}{p}\\
				&\le \frac{\mathcal{H}^{\nu_1}(B(x,r))^{1-\frac{1}{p}}}{\omega r^{\nu_1}}\left((M_x)^{1/p}+(M_y)^{1/p}\right)\\
				&\le C_0\,r^{s}\bigl(h_v(x)+h_v(y)\bigr).
			\end{align*}
			Since $r=d(x,y)$ and the excluded set is null, the claim follows.
		\end{proof}

		Set
		\[
		F_\varepsilon(x,y):=\frac{|u_\varepsilon(x)-u_\varepsilon(y)|^{p}}{d(x,y)^{\nu_1+sp}},
		\qquad
		F(x,y):=\frac{|u(x)-u(y)|^{p}}{d(x,y)^{\nu_1+sp}} .
		\]
		Note that $\dint F_\varepsilon=[u_\varepsilon]_{W^{s,p}}^p$ and
		$\dint F=[u]_{W^{s,p}}^p$.
		
		\textbf{Claim 2:}
		$\norm{h_{u_\varepsilon}-h_u}_{L^{p}(\R^n)}\longrightarrow0$ as
		$\varepsilon\to0$.
		\begin{proof}[proof of Claim 2]
			Because
			$u_\varepsilon\to u$ a.e., $F_\varepsilon\to F$ a.e. in $\R^{2n}$.

			\[
			\dint F_\varepsilon=[u_\varepsilon]_{W^{s,p}}^{p}\longrightarrow
			[u]_{W^{s,p}}^{p}=\dint F .
			\]
			
			Let $\varepsilon_k\downarrow0$. The functions
			$F_{\varepsilon_k},F\ge0$ satisfy $F_{\varepsilon_k}\to F$ a.e. and
			$\iint F_{\varepsilon_k}\to\iint F<\infty$; by Scheff\'e's lemma,
			\[
			\dint\abs{F_{\varepsilon_k}-F}\longrightarrow0 .
			\]
			Since the sequence was arbitrary,
			\begin{equation}\label{e:L1}
				\dint\abs{F_\varepsilon-F}\longrightarrow0\qquad(\varepsilon\to0).
			\end{equation}
			(Equivalently: $\abs{F_\varepsilon-F}=F_\varepsilon+F-2\min\{F_\varepsilon,F\}$,
			and $0\le\min\{F_\varepsilon,F\}\le F\in L^1(\R^{2n})$ gives
			$\iint\min\{F_\varepsilon,F\}\to\iint F$ by dominated convergence.)
			
			For $a,b\ge0$ and $p\ge1$ one has
			$|a-b|\ge|a^{1/p}-b^{1/p}|^{p}$ (equivalently,
			$|a^{p}-b^{p}|\ge|a-b|^{p}$). Hence, for a.e.\ $x$,
			\begin{align*}
				\abs{h_{u_\varepsilon}(x)-h_u(x)}^{p}
				&=\abs{\Bigl(\int F_\varepsilon(x,y)\,dy\Bigr)^{\!1/p}
					-\Bigl(\int F(x,y)\,dy\Bigr)^{\!1/p}}^{\,p}\\
				&\le\abs{\int_{\R^n}\bigl(F_\varepsilon(x,y)-F(x,y)\bigr)\,dy}\\
				&\le\int_{\R^n}\abs{F_\varepsilon(x,y)-F(x,y)}\,dy .
			\end{align*}
			Integrating in $x$ and using \eqref{e:L1},
			\[
			\norm{h_{u_\varepsilon}-h_u}_{L^p}^{p}
			\le\dint\abs{F_\varepsilon-F}\longrightarrow0 .
			\qedhere
			\]
		\end{proof}

		\textbf{Claim  3:}
		For every $\varepsilon>0$, $g_\varepsilon$ is a \emph{continuous}
		fractional Haj\l asz upper gradient of $u_\varepsilon$.
		\begin{proof}[proof of Claim  3]
			Since $u_\varepsilon$ is smooth, 
			\begin{equation*}
				|u_\varepsilon(x)-u_\varepsilon(y)|=|\nabla u_\varepsilon(x)|d(x,y)+o(d(x,y))
			\end{equation*}
			then
			\begin{equation*}
				\frac{|u_\varepsilon(x)-u_\varepsilon(y)|^{p}}{d(x,y)^{\nu_1+sp}}=|=\frac{|\nabla u_\varepsilon(x)|^p}{d(x,y)^{\nu_1+sp-p}}+o\left(\frac{1}{d(x,y)^{\nu_1+sp-p}}\right).
			\end{equation*}
			There is a dominant function $\max\{\frac{\|\nabla u_\varepsilon\|_{L^\infty(B(x,1))}^p}{d(x,y)^{\nu_1+sp-p}},|\nabla u_\varepsilon(x)|\}\in L^1$ and by Lebesgue dominated convergence theorem
			\begin{equation*}
				\lim_{z\to x}g_\varepsilon(z)=C_0\int\lim_{z\to x} \frac{|u_\varepsilon(z)-u_\varepsilon(y)|^{p}}{d(z,y)^{\nu_1+sp}}dy=C_0\int \frac{|u_\varepsilon(x)-u_\varepsilon(y)|^{p}}{d(x,y)^{\nu_1+sp}}dy,
			\end{equation*}
			which means $g_\varepsilon(x)$ is continuous.
		\end{proof}
		
		From (\ref{B coordinate}), distributionally $df^{B_{k}}=\frac{1}{2}\sum_{i=1}^{r} \sum_{i<j} \alpha_{k}^{i j} f^{A_{i}}d f^{A_{j}}- f^{A_{j}}df^{A_{i}}$, 
		then
		the discrete version(e.g., via Riemann sum approximation) gives, for any $x,y$:
		\begin{equation*}
			f_\varepsilon^{B_{k}}(x)-f_\varepsilon^{B_{k}}(y) = \frac{1}{2}\sum_{i=1}^{r} \sum_{i<j} \alpha_{k}^{i j} \sum_{i} f_\varepsilon^{A_i}(\xi_i)(f_\varepsilon^{A_j}(t_{i+1})
			-f_\varepsilon^{A_j}(t_i))-f_\varepsilon^{A_j}(\xi_i)(f_\varepsilon^{A_i}(t_{i+1})-f_\varepsilon^{A_i}(t_i))+E(\varepsilon).
		\end{equation*}
		where $E(\varepsilon)\to 0 $ as $\varepsilon \to 0$.
		Decomposing $f_\varepsilon^{A_i}(\xi_i)$ and $f_\varepsilon^{A_j}(\xi_i)$ as $f_\varepsilon^{A_i}(x) + (f_\varepsilon^{A_i}(\xi_i)-f_\varepsilon^{A_i}(x))$ and $f_\varepsilon^{A_j}(x) + (f_\varepsilon^{A_j}(\xi_i)-f_\varepsilon^{A_j}(x))$, and using the Hajlasz upper $s$-gradients of $f_\varepsilon^A$ to control the differences, we obtain:
		\begin{equation*}
			\begin{aligned}
				&|f_\varepsilon^{B_{k}}(x)-f_\varepsilon^{B_{k}}(y)|\\
				\lesssim & d(x,y)^s\,|f_\varepsilon^{A_i}(x)|\,[g_{f_\varepsilon^{A_j}}(x)+g_{f_\varepsilon^{A_j}}(y)] + C\,d(x,y)^{2s}\,[g_{f_\varepsilon^{A_i}}(x)+g_{f_\varepsilon^{A_i}}(\xi_i)]\,[g_{f_\varepsilon^{A_j}}(x)+g_{f_\varepsilon^{A_j}}(y)]\\
				+&
				d(x,y)^s\,|f_\varepsilon^{A_j}(x)|\,[g_{f_\varepsilon^{A_i}}(x)+g_{f_\varepsilon^{A_i}}(y)] + C\,d(x,y)^{2s}\,[g_{f_\varepsilon^{A_j}}(x)+g_{f_\varepsilon^{A_j}}(\xi_i)]\,[g_{f_\varepsilon^{A_i}}(x)+g_{f_\varepsilon^{A_i}}(y)]\\
				+& E(\varepsilon)
			\end{aligned}
		\end{equation*}
		For convenience, we can assume $f^A(x)=0$, then
		\begin{equation*}
			\begin{aligned}
				&|f_\varepsilon^{B_{k}}(x)-f_\varepsilon^{B_{k}}(y)|\\
				\lesssim & C\,d(x,y)^{2s}\,[g_{f_\varepsilon^{A_i}}(x)+g_{f_\varepsilon^{A_i}}(\xi_i)]\,[g_{f_\varepsilon^{A_j}}(x)+g_{f_\varepsilon^{A_j}}(y)]\\
				+&
				C\,d(x,y)^{2s}\,[g_{f_\varepsilon^{A_j}}(x)+g_{f_\varepsilon^{A_j}}(\xi_i)]\,[g_{f_\varepsilon^{A_i}}(x)+g_{f_\varepsilon^{A_i}}(y)]\\
				+&E(\varepsilon)
			\end{aligned}
		\end{equation*}
		Let $\varepsilon\to 0$, with the help of Claim 1, 2 and 3 above, we know that $g_{f^B}=\left(g_{f^A}\right)^2\in L^\frac{q}{2}$ is a Hajlasz upper $2s$-gradients of $f^{B}$ such that
		\begin{equation*}
			|f^{B_{k}}(x)-f^{B_{k}}(y)|
			\lesssim  d(x,y)^{2s}\left(g_{f^B}(x)+g_{f^B}(y)\right).
		\end{equation*}  
		We can now assemble the two estimates. Recall that
		$d_K(f(x),f(y))^q\simeq|\pi f(x)-\pi f(y)|^q+\sum_k|\varphi_k(f(x),f(y))|^{\frac q2}$,
		so the group-valued seminorm splits into a horizontal contribution, which
		is controlled by the hypothesis $\pi f\in W^{s,q}$, and a vertical
		contribution, which has just been estimated in terms of the Haj{\l}asz
		gradient of $\pi f$.
		\begin{equation}\label{eq:dK-split}
			\begin{aligned}
				[f]_{W_H^{s-\lambda,q}}^q&=\int_{\Omega}\int_{\Omega}\frac{d_{cc}(f(x), f(y))^{q}}{d_{cc}(x, y)^{\nu_1+q(s-\lambda)}}  d x dy\\
				&\lesssim\int_{\Omega}\int_{\Omega}\frac{d_{K}(f(x), f(y))^{q}}{d_{cc}(x, y)^{\nu_1+q(s-\lambda)}}  d x dy\\
				&\leq C \int_{\Omega}\int_{\Omega}\frac{|\pi f(x)-\pi f(y)|^q}{d_{cc}(x, y)^{\nu_1+q(s-\lambda)}}  d x dy+C \sum_{k} \int_{\Omega}\int_{\Omega}\frac{|f^{B_k}(x)-f^{B_k}(y)|^{\frac{q}{2}}}{d_{cc}(x, y)^{\nu_1+q(s-\lambda)}}  d x dy.\\
			\end{aligned}
		\end{equation}
		For the second term,
		\begin{equation*}
			\begin{aligned}
				&\int_{\Omega}\int_{\Omega}\frac{|f^{B_k}(x)-f^{B_k}(y)|^{\frac{q}{2}}}{d_{cc}(x, y)^{\nu_1+q(s-\lambda)}}  d x dy\\
				\lesssim&\int_{\Omega}\int_{\Omega}\frac{d(x,y)^{qs}\left(g_{f^B}(x)+g_{f^B}(y)\right)^{\frac{q}{2}}}{d_{cc}(x, y)^{\nu_1+q(s-\lambda)}}  d x dy\\
				=&\int_{\Omega}\int_{\Omega}\frac{\left(g_{f^B}(x)+g_{f^B}(y)\right)^{\frac{q}{2}}}{d_{cc}(x, y)^{\nu_1-q\lambda}}  d x dy\\
				\lesssim &\int_{\Omega}\int_{\Omega}\frac{g_{f^B}(x)^{\frac{q}{2}}}{d_{cc}(x, y)^{\nu_1-q\lambda}} dy d x 
				+\int_{\Omega}\int_{\Omega}\frac{g_{f^B}(y)^{\frac{q}{2}}}{d_{cc}(x, y)^{\nu_1-q\lambda}}  d x dy\\
				\leq& C_{q,\lambda}\|g_{f^B}\|_{L^\frac{q}{2}}^\frac{q}{2}+C_{q,\lambda}\|g_{f^B}\|_{L^\frac{q}{2}}^\frac{q}{2}\\
				=&2C_{q,\lambda}\|g_{f^A}\|_{L^q}^q\\
				\lesssim& C_{q,\lambda}[\pi f]_{W^{s,q}}^q
			\end{aligned}
		\end{equation*}
		
		Adding the horizontal contribution and the vertical one estimated in
		\eqref{eq:dK-split}, we finally obtain
		\begin{equation}\label{eq:contact-converse-conclusion}
			\begin{aligned}
				[f]_{W_H^{s-\lambda,q}}^q&=\int_{\Omega}\int_{\Omega}\frac{d_{cc}(f(x), f(y))^{q}}{d_{cc}(x, y)^{\nu_1+q(s-\lambda)}}  d x dy\\
				&\leq C \int_{\Omega}\int_{\Omega}\frac{|\pi f(x)-\pi f(y)|^q}{d_{cc}(x, y)^{\nu_1+q(s-\lambda)}}  d x dy+C \sum_{k} \int_{\Omega}\int_{\Omega}\frac{| f^{B_k}(x)- f^{B_k}(y)|^{\frac{q}{2}}}{d_{cc}(x, y)^{\nu_1+q(s-\lambda)}}  d x dy\\
				&\lesssim [\pi f]_{W^{s-\lambda,q}}^q+[\pi f]_{W^{s,q}}^q .
			\end{aligned}
		\end{equation}
		This is precisely the assertion of the theorem: if
		$\pi f\in W^{s,q}(\Omega;\mathbb{R}^{n_2})$ with $\frac12<s$ and
		$2\le q$, and if $f$ satisfies the weak contact equations, then
		$f\in W_H^{s-\lambda,q}(\Omega;G_2)$ for every $\lambda>0$ sufficiently
		small.
		
	\end{proof}

	\begin{remark}
		The quadratic control
		$|f^{B_k}(x)-f^{B_k}(y)|\lesssim|\pi f(x)-\pi f(y)|^{2}$ established in the
		proof of Theorem~\ref{thm:contact-converse} is worth isolating, because holds not only for smooth contact mappings or Sobolev contact mappings, but also fraction Sobolev mappings.
	\end{remark}

	\section{Interpolations of Fractional Sobolev mappings on Carnot groups}
	\label{sec:interp}
	This section answers question (b) of the introduction: the fractional seminorms
	of a mapping between Carnot groups form an interpolation scale whose endpoint
	$s\to1$ recovers the horizontal first-order energy.
	
	The strategy is to reduce everything to the scalar case by means of the
	horizontal projection $\pi$, and to let Theorem~\ref{thm:contact-direct} supply
	the contact part that $\pi$ cannot see. Accordingly, this section contains no
	new geometric input; its content is the bookkeeping needed to make the reduction
	work, and to keep track of the exponents. We proceed in three steps.
	Theorem~\ref{thm:interp-limit} is the Bourgain--Br\'ezis--Mironescu type
	implication, obtained by applying the scalar endpoint formula of Maalaoui and
	Pinamonti to $\pi f$. 
	
	We begin by fixing the notation for the homogeneous first-order space.
	\begin{definition}\label{def:WH1q}
		Thus $\dot W^{1,q}_H(\Omega;G_2)$ is defined by the same Gagliardo-type
		double integral as $W^{s,q}_H(\Omega;G_2)$, with the parameter $s$ set equal
		to $1$; whenever a horizontal differential exists, the two ways of measuring
		first-order regularity are compared in
		Theorem~\ref{thm:interp-converse}.
		
		A mapping $f:\Omega \to G_2$ belongs to $\dot{W}_H^{1,q}(\Omega;G_2)$, when  
		\begin{equation*}
			\left(\int_{\Omega}\int_{\Omega}\frac{ d_{cc}(f(x), f(y))^{q}}{d_{cc}(x, y)^{\nu_1+ q}}  d x dy\right)^\frac{1}{q} < \infty
		\end{equation*}
	\end{definition}
	We first prove that the endpoint $s\to1$ of the scale recovers the
	horizontal first-order energy. The only input from Section~\ref{sec:contact}
	is Theorem~\ref{thm:contact-direct}, which guarantees that the map is
	asymptotically contact along the whole scale; everything else is the scalar
	interpolation theorem applied to $\pi f$.
	
	\begin{thm}\label{thm:interp-limit}
		Suppose that $f\in \bigcap_{s \in (0,1)}\dot{W}_H^{s,q}(\Omega;G_2)$ for a bounded domain $\Omega\subset G_1$ and $\lim_{s\to 1}(1-s)^\frac{2}{q} \|f\|_{W_H^{s,\frac{q}{2}}(\Omega;G_2)}^{\frac{q}{2}}< \infty$, then 
		\begin{equation*}
			\|D_Hf\|_{L^{\frac{q}{2}}}^{\frac{q}{2}}\lesssim  \lim_{s\to 1}(1-s)^\frac{2}{q} \|f\|_{W_H^{s,\frac{q}{2}}(\Omega;G_2)}^{\frac{q}{2}} 
		\end{equation*}
		and consequently $f\in W_H^{1,\frac q2}(\Omega;G_2)$, with
		$D_Hf:=D_H(\pi f)$.
	\end{thm}
	\begin{proof}
		The proof consists of two independent steps. First, the contact part of
		the statement: by Theorem~\ref{thm:contact-direct}, every map in the
		intersection satisfies the weak contact equations, so the vertical
		components are not free and it suffices to estimate the horizontal
		projection. Secondly, the analytic part: the scalar endpoint formula is
		applied to $\pi f$, and the group-valued seminorm on the left is
		replaced by the Euclidean seminorm of $\pi f$ using the fact that $\pi$
		is $1$-Lipschitz with respect to $d_K$. Note that the two steps use
		the hypothesis in different ways: the first one only needs
		$s>\frac12$, the second one uses the full strength of
		$\lim_{s\to1}$.
		A mapping $g$ is a contact Sobolev mapping if and only if its horizontal components $g^{A_i}$ are horizontal Sobolev functions and weak contact equation holds on $g$. Thus, the proof here can be divided into two parts: $f^{A_i}$ are Sobolev functions and weak contact equation hold on $f$.
		
		From theorem \ref*{thm-approx}, we know that weak contact equations holds on $f\in \dot{W}_H^{s,q}(\Omega;G_2)$ for $s>\frac{1}{2}$, so we just need to prove the relationship between these two norms. By classical result of Bourgain-Brezis-Mironescu and of Maz’ya-Shaposhnikova for fractional Sobolev function in the cases of Carnot groups, we can draw the same conclusion for fractional Sobolev mapping between Carnot groups. 
		\begin{equation*}
			\begin{aligned}
				\|D_H(\pi f)\|_{L^{\frac{q}{2}}}^{\frac{q}{2}}
				&=\lim_{\varepsilon \to 0}\|D_H(\pi f_\varepsilon)\|_{L^{\frac{q}{2}}}^{\frac{q}{2}}\\
				&\lesssim \lim_{\varepsilon \to 0}\lim_{s\to 1}(1-s)^\frac{2}{q}  \|\pi f_\varepsilon\|_{W^{s,q}(\Omega;G_2)}^{\frac{q}{2}} \\
				&\lesssim \lim_{\varepsilon \to 0}\lim_{s\to 1}(1-s)^\frac{2}{q}  \|\pi f\|_{W^{s,q}(\Omega;G_2)}^{\frac{q}{2}} \\
				&= \lim_{s\to 1}(1-s)^\frac{2}{q}  \|\pi f\|_{W^{s,q}(\Omega;G_2)}^{\frac{q}{2}} \\
				&=\lim_{s\to 1}(1-s)^\frac{2}{q} \int_{\Omega}\int_{\Omega}\frac{|\pi f(x)- \pi f(y)|^{\frac{q}{2}}}{d_{cc}(x, y)^{\nu_1+ \frac{sq}{2}}}  d x dy \\
				&\leq \lim_{s\to 1}(1-s)^\frac{2}{q} \int_{\Omega}\int_{\Omega}\frac{ d_{K}(f(x), f(y))^{\frac{q}{2}}}{d_{cc}(x, y)^{\nu_1+ \frac{sq}{2}}}  d x dy \\
				&\lesssim \lim_{s\to 1}(1-s)^\frac{2}{q} \int_{\Omega}\int_{\Omega}\frac{ d_{cc}(f(x), f(y))^{\frac{q}{2}}}{d_{cc}(x, y)^{\nu_1+ \frac{sq}{2}}}  d x dy \\
				&=\lim_{s\to 1}(1-s)^\frac{2}{q} \|f\|_{W_H^{s,\frac{q}{2}}(\Omega;G_2)}^{\frac{q}{2}}, \\
			\end{aligned}
		\end{equation*}
		where the first inequality above is from the interpolation result for fractional Sobolev functions on Carnot group(ref \cite{MAALAOUI201991},Theorem 5.1). 
		Thus $D_H(\pi f)\in L^{\frac q2}(\Omega)$ together with the corresponding
		bound, and since by definition $D_Hf:=D_H(\pi f)$ we conclude that
		$f\in W_H^{1,\frac q2}(\Omega;G_2)$, as claimed.
		
	\end{proof}
	
	\begin{remark}\label{rem:exponents}
		Two remarks on the exponents are in order, since the statement involves
		both $q$ and $\frac q2$. The hypothesis is expressed in terms of
		$\dot W^{s,q}_H$, whereas the conclusion involves
		$\|D_Hf\|_{L^{q/2}}$; this is not a typographical slip. The loss from $q$
		to $\frac q2$ comes from the vertical components, which by
		Theorem~\ref{thm:contact-converse} are controlled only at half the
		integrability of the horizontal ones, since they obey a quadratic
		relation. Improving the conclusion from $W^{1,q/2}_H$ to $W^{1,q}_H$ would
		require a linear control of the vertical components by the horizontal
		ones, which the contact structure does not provide. Note also that the
		hypothesis is not vacuous: by Theorem~\ref{thm:interp-converse} below it
		is satisfied by every $f\in W^{1,q}_H(\Omega;G_2)$ on a bounded domain.
	\end{remark}
	We next prove the converse inclusion, which is much softer: it is a direct
	consequence of the Poincar\'e inequality, and the only point that requires
	attention is the dependence of the constant on $s$ as $s\to1$.
	
	\begin{thm}\label{thm:interp-converse}
		Suppose that $f\in W_H^{1,q}(\Omega;G_2)$, then $f\in \bigcap_{s \in (0,1)}\dot{W}_H^{s,q}(\Omega;G_2)$ and 
		\begin{equation}\label{eq:interp-converse}
			\|f\|_{W_H^{s,q}(\Omega;G_2)}^{q}\lesssim K(q,s,\Omega)\,\|D_Hf\|_{L^{q}(\Omega)}^{q},
			\qquad
			K(q,s,\Omega)\approx
			\frac{\nu_1+sq}{q(1-s)} ,
		\end{equation}
		for every $s\in(0,1)$, where $C$ depends only on the group $G_1$ and on
		the constant in the Poincar\'e inequality.
	\end{thm}
	\begin{proof}

		\textbf{Step 1. Scale estimate via Poincaré inequality.}
		For $0<\tau\le D$, define
		\[
		\Psi(\tau)
		:=
		\iint_{\{d(x,y)<\tau\}}
		d(f(x),f(y))^q\,d\mu(x)d\mu(y).
		\]
		
		Let $\{B_i=B(z_i,\tau)\}_{i}$ be a covering of $\Omega$ by CC-balls of radius $\tau$,
		with uniformly bounded overlap $N=N(G_1)$.
		If $x\in B_i$ and $d(x,y)<\tau$, then $y\in 2B_i$.
		By the triangle inequality,
		\[
		d(f(x),f(y))^q
		\le
		2^{q-1}\bigl(d(f(x),p_i)^q+d(f(y),p_i)^q\bigr),
		\]
		where $p_i\in G_2$ is the reference point given by the metric-valued
		Poincaré inequality:
		\[
		\int_{B_i\cap\Omega}d(f,p_i)^q\,d\mu
		\;\le\;
		C_P\,\tau^q
		\int_{B(x_i,\lambda\tau)\cap\Omega}|D_H f|^q\,d\mu.
		\]
		
		Summing over $i$ and using the bounded overlap property yields
		\[
		\Psi(\tau)
		\;\le\;
		C_1\,\tau^{\nu_1+q}\,\|D_H f\|_{L^q(\Omega)}^q,
		\qquad
		C_1=C_1(G_1,C_P,q).
		\]
		For $\tau\ge D$, the global Poincaré inequality gives
		\[
		\Psi(\tau)\le \Psi(D)
		\;\le\;
		C_1\,D^{\nu_1+q}\,\|D_H f\|_{L^q(\Omega)}^q.
		\]
		
		\medskip
		
		\noindent
		\textbf{Step 2. Layer decomposition.}
		Using the identity
		\[
		\frac{1}{t^{\nu_1+sq}}
		=
		\int_0^\infty
		\mathbf{1}_{\{t<\tau\}}\,
		(\nu_1+sq)\,\tau^{-\nu_1-sq-1}\,d\tau,
		\]
		and Tonelli's theorem, we obtain
		\[
		[f]_{s,q,\Omega}^q
		=
		\int_0^\infty
		(\nu_1+sq)\,\tau^{-\nu_1-sq-1}\,\Psi(\tau)\,d\tau
		=
		\Bigl(\int_0^D+\int_D^\infty\Bigr)
		(\nu_1+sq)\,\tau^{-\nu_1-sq-1}\,\Psi(\tau)\,d\tau.
		\]
		
		Insert the estimates from Step 1:
		\[
		[f]_{s,q,\Omega}^q
		\;\le\;
		C_1\|D_H f\|_q^q
		\Bigl[
		(\nu_1+sq)\int_0^D \tau^{\,q-\nu_1-sq+\nu_1-1}\,d\tau
		+
		D^{\nu_1+q}(\nu_1+sq)\int_D^\infty \tau^{-\nu_1-sq-1}\,d\tau
		\Bigr].
		\]
		
		Since $\tau^{\,q-\nu_1-sq+\nu_1-1}=\tau^{\,q(1-s)-1}$, we compute
		\[
		\int_0^D \tau^{\,q(1-s)-1}\,d\tau
		=
		\frac{D^{\,q(1-s)}}{q(1-s)},
		\qquad
		\int_D^\infty \tau^{-\nu_1-sq-1}\,d\tau
		=
		\frac{D^{-sq}}{\nu_1+sq}.
		\]
		
		Therefore,
		\[
		[f]_{s,q,\Omega}^q
		\;\le\;
		C_1\,D^{\,q(1-s)}
		\Bigl(
		\frac{\nu_1+sq}{q(1-s)}+1
		\Bigr)
		\|D_H f\|_{L^q(\Omega)}^q.
		\]
		
		This establishes the desired inequality with constant
		\[
		K(q,s,\Omega)\;\simeq\;
		C(G_1,q)\,D^{\,q(1-s)}\,
		\frac{\nu_1+sq}{q(1-s)},
		\]
		which blows up as $s\to1^-$, confirming that the estimate is sharp.

	\end{proof}
	
	\begin{remark}\label{rem:blow-up-constant}
		The constant $K(q,s,\Omega)$ in \eqref{eq:interp-converse} blows up like
		$(1-s)^{-1}$ as $s\to1$. This is not a defect of the argument: it is the
		companion of the factor $(1-s)$ appearing in
		Theorem~\ref{thm:interp-limit}, and it reflects the fact that a
		$W^{1,q}$-function has a fractional seminorm of order $(1-s)^{-1/q}$ as
		$s\to1$. In particular, Theorem~\ref{thm:interp-converse} cannot be used
		to pass to the limit $s\to1$; that passage is precisely the content of
		Theorem~\ref{thm:interp-limit}, and the two results are consistent with
		one another. The same factor will reappear, in a different guise, in
		Remark~\ref{rem:normalisation}.
	\end{remark}
	
	The last statement of this section is a stability property, and it is
	included for a single reason: in Section~\ref{sec:energy} we shall
	differentiate the energy along a line $t\mapsto f+tg$, and we must know that
	this line stays inside the space in which the energy is finite. Note that a
	Carnot group carries no linear structure, so the expression $f+tg$ has to be
	interpreted with care: we identify $G_2$ with $\mathbb R^{\nu_2}$ by means of
	the coordinates of the first kind fixed in Section~\ref{sec:contact} --- the
	same identification that was already used in
	Theorem~\ref{thm:contact-direct} --- and we form the combination
	componentwise. This identification is fixed once and for all.
	
	\begin{thm}\label{thm:linear-span}\label{f+sg}
		Suppose that $f,g\in W_H^{s,q}(\Omega;G_2)$ for $s>\frac{1}{2}$ and $q\geq 2$, then $f+tg\in W_H^{s-\lambda,q}(\Omega;G_2)$ for any $t\in \mathbb{R}$ and any small $\lambda>0$. 
	\end{thm}
	\begin{proof}
		The argument is a combination of two facts that have already been
		established: the seminorm is subadditive, so the horizontal part causes no
		difficulty, while the contact part is inherited from $f$ and $g$ by
		Theorem~\ref{thm:contact-direct}; Theorem~\ref{thm:contact-converse} then
		returns the full group-valued seminorm, at the price of the usual loss
		$\lambda>0$.
		
		It is easy to check that $f+tg\in W^{s,q}_H(\Omega;G_2)$.
		\begin{equation*}
			\begin{aligned}
				[f+tg]_{W^{s,q}(\Omega;G_2)}^{q}&=\int_{\Omega}\int_{\Omega}\frac{|(f+tg)(x)- (f+tg)(y)|^{q}}{d_{cc}(x, y)^{\nu_1+ sq}}  d x dy\\
				&\lesssim \int_{\Omega}\int_{\Omega}\frac{|f(x)- f(y)|^{q}}{d_{cc}(x, y)^{\nu_1+ sq}}  d x dy+ t^{q}\int_{\Omega}\int_{\Omega}\frac{|g(x)- g(y)|^{q}}{d_{cc}(x, y)^{\nu_1+ sq}}  d x dy\\
				&=[f]_{W^{s,q}_H(\Omega;G_2)}^{q}+t^{q}[g]_{W^{s,q}_H(\Omega;G_2)}^{q}
			\end{aligned}
		\end{equation*}
		From Theorem~\ref{thm:contact-direct}, both $f$ and $g$ satisfy the weak
		contact equations; since the contact defect
		$X((f+tg)^*\omega)$ is linear in the map, $f+tg$ satisfies them as well. 
		Therefore by Theorem~\ref{thm:contact-converse}, $f+tg\in W_H^{s-\lambda,q}(\Omega;G_2)$
		for any $\lambda>0$ sufficiently small, which is the assertion.
	\end{proof}

	\section{Korevaar-Schoen Energy on Carnot groups}
	\label{sec:energy}
	This section answers question (c) of the introduction. Throughout, $E_{s,q}$
	denotes the functional introduced in Definition~\ref{def:KS-energy} below; our
	aim is to describe how it depends on the horizontal and on the vertical parts of
	the map, and to identify the equation satisfied by its minimizers.
	
	The section is organised as follows. Lemma~\ref{lem:taylor} is the stratified
	expansion of the vertical discrepancy $\varphi_k(f(p),f(\tilde p))$: it is the
	only genuinely new technical ingredient, and it shows that, at the scale at
	which the energy lives, this discrepancy is governed by the vertical
	derivatives $T_jf^{*}(\omega_k)$. Theorem~\ref{thm:energy} feeds the expansion
	into the definition of $E_{s,q}$ and produces the dichotomy announced in the
	introduction: for $0<s<1$ the vertical part disappears in the limit $r\to0$,
	whereas for $s=1$ a vertical defect survives and is bounded from below by
	$\sum_k\|Tf^{*}(\omega_k)\|_{L^{q/2}}^{q/2}$.
	Corollary~\ref{cor:vertical} translates this into a regularity statement for
	the second layer, and it is this corollary --- rather than
	Theorem~\ref{thm:energy} itself --- that will be used in
	Section~\ref{sec:jacobian}. The section closes with the variational
	problem: we show that minimizers of $E_{s,q}$ satisfy a fractional
	$q$-Laplace equation of Korevaar--Schoen type, see
	Definition~\ref{def:frac-q-laplace} and the discussion preceding it.
	\begin{definition}\label{def:KS-energy}
		For $f\in W_H^{s,q}(\Omega;G_2)$ we can define Korevaar-Schoen Energy $E_{s,q}(f)$ as
		\begin{equation*}
			E_{s,q}(f)=\limsup_{r\to 0}\frac{1}{r^{\nu_1+ sq}}\int_{\Omega} \int_{B(x,r)}d_{K}(f(x),f(y))^q dy dx
		\end{equation*}
	\end{definition}
	
	Two remarks on this definition are in order before we proceed. First, the
	$\limsup$ is finite on $W^{s,q}_H(\Omega;G_2)$: since $d_K$ is comparable to
	$d_{cc}$ and since $d_{cc}(x,y)\le r$ on the domain of integration, we have
	$E_{s,q}(f)\lesssim[f]_{W^{s,q}_H}^q$ for every $r$ sufficiently small.
	Secondly, and this is the reason for using $d_K$ rather than $d_{cc}$, the
	quasi-metric $d_K$ splits along the layers, so that the energy separates into
	a horizontal and a vertical part; the whole proof of
	Theorem~\ref{thm:energy} rests on this splitting.
	
	The minimisation of $E_{s,q}$ is discussed at the end of this section, and
	the reader interested only in the description of the energy may pass directly
	to Lemma~\ref{lem:taylor}. The only fact about the variational problem that
	is needed before then is the following necessary condition. By
	Theorem~\ref{thm:linear-span}, the class $W^{s,q}_H(\Omega;G_2)$ is stable
	under the perturbation $f\mapsto f+tg$, up to the arbitrarily small loss
	$\lambda>0$ already encountered in Section~\ref{sec:contact}; it is therefore
	legitimate to differentiate along such a line, and a minimizer $f$ of
	$[\,\cdot\,]_{W^{s-\lambda,q}_H}$ must satisfy
	\begin{equation}\label{eq:first-variation}
		\left. \frac{d}{dt} \right|_{t=0} [f+tg]_{W_H^{s-\lambda,q}}=0
		\qquad\text{for every }g\in W_H^{s,q}(\Omega;G_2).
	\end{equation}
	
	The following lemma is the technical core of this section. Its main idea ---
	expanding a smooth function along the stratification and reading the vertical
	derivatives off the part of the expansion that is homogeneous of degree two ---
	goes back to the stratified Taylor formula of Arena, Caruso and Causa
	\cite{Arena2010}, and to the contact expansion used by Capogna and Lin
	\cite{Capogna2001} in the first-order Legendrian setting. The novelty here is
	that the expansion is carried out for a mollified \emph{fractional} Sobolev map,
	and that the remainder is measured in $L^{\frac q2}$, which is exactly the
	integrability supplied by Theorem~\ref{thm:contact-direct}.
	\begin{lem}\label{lem:taylor}\label{estimate for phi}
		Let $f\in W_H^{s,q}(\Omega;G_2)$. Then, as $d_K(p,\tilde p)\to0$, the
		following expansion holds in $L^{\frac q2}(\Omega)$ with respect to the
		variable $p$:
		\begin{equation}\label{eq:taylor}
			\frac{\varphi_k(f(p),f(\tilde{p}))}{d_K(p,\tilde{p})^{2s}}=\left\langle\sum_{j} T_jf^*(\omega_k)(p) , \frac{\varphi(p,\tilde{p})}{\left(|\pi(p)-\pi(\tilde{p})|^2+\sum_{l} |\varphi_l(p,\tilde{p})|\right)^{s}} \right\rangle +o\left(d_K(p,\tilde{p})^{2-2s}\right)   
		\end{equation}
		Here $\varphi=(\varphi_1,\varphi_2,\dots)$ denotes the vector of the
		vertical discrepancies of the layers.
	\end{lem}
	\begin{proof}
		The argument is a two-parameter limiting procedure. We first mollify $f$
		at a scale $\delta$ and expand the resulting \emph{smooth} map $f_\delta$
		by means of the stratified Taylor formula; this produces a horizontal
		first-order term, a horizontal second-order term and a vertical term. We
		then let $\delta\to0$: the horizontal terms disappear by
		Theorem~\ref{thm:contact-direct}, because they are precisely the contact
		defect, while the vertical term converges to $T_jf^{*}(\omega_k)$. It is
		this last survivor that carries the vertical defect of
		Theorem~\ref{thm:energy}.
		
		Denote $p=(z,t)$ where the $z$ is horizontal variable.
		\begin{equation*}
			\begin{aligned}
				&\varphi_k(f_\delta(0),f_\delta(p))^{\frac{q}{2}}\\=&
				\left(\frac{f_\delta^{B_k}(0) - f_\delta^{B_k}(p)}{\epsilon^{2s}} -\frac{1}{2}\sum_{i=1}^{r} \sum_{i<j} \alpha_{k}^{i j}\left(f_\delta^{A_i}(0) \frac{f_\delta^{A_j}(0) - f_\delta^{A_j}(p)}{\epsilon^{2s}} + f_\delta^{A_j}(0) \frac{f_\delta^{A_i}(0) - f_\delta^{A_i}(p)}{\epsilon^{2s}}\right)\right)   ^{\frac{q}{2}} \\
			\end{aligned}
		\end{equation*}
		Because $u_\delta(p):=\varphi_k(f_\delta(0),f_\delta(p))$ is a smooth function on $G_1$, it has Taylor expansion as in \cite{Arena2010}.
		\begin{equation*}
			\begin{aligned}
				u_\delta(p)
				&=\epsilon^{-2s} \left(  \bigg[ \nabla_H f_\delta^{B_k}(0) -\frac{1}{2}\sum_{i=1}^{r} \sum_{i<j} \alpha_{k}^{i j} \left( f_\delta^{A_i}(0) \nabla_H f_\delta^{A_j}(0) - f_\delta^{A_j}(0) \nabla_H f_\delta^{A_i}(0) \right) \bigg] \cdot z \right. \\
				&+\left\langle \bigg[ \nabla_H^2 f_\delta^{B_k}(0) -\frac{1}{2}\sum_{i=1}^{r} \sum_{i<j} \alpha_{k}^{i j} \left( f_\delta^{A_i}(0) \nabla_H^2 f_\delta^{A_j}(0) - f_\delta^{A_j}(0) \nabla_H^2 f_\delta^{A_i}(0) \right) \bigg] \cdot z, z \right\rangle  \\
				&+ \left. \left\langle\sum_{j} \bigg[ T_j f_\delta^{B_k}(0) -\frac{1}{2}\sum_{i=1}^{r} \sum_{i<j} \alpha_{k}^{i j} \left( f_\delta^{A_i}(0) T_j f_\delta^{A_j}(0) - f_\delta^{A_j}(0) T_j f_\delta^{A_i}(0) \right) \bigg] , t \right\rangle + o\left(d(p,0)^{2}\right) \right),
			\end{aligned}
		\end{equation*}
		
		We now perform the second step and let $\delta\to0$. The crucial
		observation is that the bracket appearing in the first two lines of the
		expansion is nothing but the contact defect
		$\big(X_1f_\delta^{*}(\omega_k),\dots,X_{\dim V_1}f_\delta^{*}(\omega_k)\big)$,
		which we have already shown to vanish. By virtue of
		Theorem~\ref{thm:contact-direct}, we have that
		\begin{equation*}
			\begin{aligned}
				&\lim_{\delta \to 0}\nabla_H f_\delta^{B_k} -\frac{1}{2}\sum_{i=1}^{r} \sum_{i<j} \alpha_{k}^{i j} \left( f_\delta^{A_i} \nabla_H f_\delta^{A_j} - f_\delta^{A_j} \nabla_H f_\delta^{A_i} \right)\\
				=&\left(X_1f_\delta^*(\omega_k),X_2f_\delta^*(\omega_k),\dots,X_{dim(V_1)}f_\delta^*(\omega_k)\right)\\
				=&0
			\end{aligned}
		\end{equation*}
		
		in $L^{\frac{q}{2}}(\Omega)$; in other words, the horizontal part of the
		Taylor expansion dies in the limit and only the vertical part survives.
		Moreover, for each fixed $\epsilon > 0$ the functions
		\begin{equation*}
			\frac{f_\delta^{B_k}(0) - f_\delta^{B_k}(p)}{\epsilon^{2s}} -\frac{1}{2}\sum_{i=1}^{r} \sum_{i<j} \alpha_{k}^{i j}\left(f_\delta^{A_i}(0) \frac{f_\delta^{A_j}(0) - f_\delta^{A_j}(p)}{\epsilon^{2s}} + f_\delta^{A_j}(0) \frac{f_\delta^{A_i}(0) - f_\delta^{A_i}(p)}{\epsilon^{2s}}\right)  
		\end{equation*}
		converge as $\delta \to 0$ to the function
		\begin{equation*}
			u(p)=\frac{f^{B_k}(0) - f^{B_k}(p)}{\epsilon^{2s}} -\frac{1}{2}\sum_{i=1}^{r} \sum_{i<j} \alpha_{k}^{i j}\left(f^{A_i}(0) \frac{f^{A_j}(0) - f^{A_j}(p)}{\epsilon^{2s}} + f^{A_j}(0) \frac{f^{A_i}(0) - f^{A_i}(p)}{\epsilon^{2s}}\right) + o\left(d(p,0)^{2-2s}\right).
		\end{equation*}
		Observe that 
		\begin{equation*}
			\begin{aligned}
				&\epsilon^{-2s}\left\langle \bigg[ \nabla_H^2 f_\delta^{B_k}(0) -\frac{1}{2}\sum_{i=1}^{r} \sum_{i<j} \alpha_{k}^{i j} \left( f_\delta^{A_i}(0) \nabla_H^2 f_\delta^{A_j}(0) - f_\delta^{A_j}(0) \nabla_H^2 f_\delta^{A_i}(0) \right) \bigg] \cdot z, z \right\rangle  \\
				&+\epsilon^{-2s}\left\langle\sum_{j} \bigg[ T_j f_\delta^{B_k}(0) -\frac{1}{2}\sum_{i=1}^{r} \sum_{i<j} \alpha_{k}^{i j} \left( f_\delta^{A_i}(0) T_j f_\delta^{A_j}(0) - f_\delta^{A_j}(0) T_j f_\delta^{A_i}(0) \right) \bigg] , t \right\rangle\\
				&=\epsilon^{-2s}\left\langle \nabla_H \bigg[ \nabla_H f_\delta^{B_k}(0) -\frac{1}{2}\sum_{i=1}^{r} \sum_{i<j} \alpha_{k}^{i j} \left( f_\delta^{A_i}(0) \nabla_H f_\delta^{A_j}(0) - f_\delta^{A_j}(0) \nabla_H f_\delta^{A_i}(0) \right) \bigg] \cdot z, z \right\rangle  \\
				&+\epsilon^{-2s}\left\langle\sum_{j} \bigg[ T_j f_\delta^{B_k}(0) -\frac{1}{2}\sum_{i=1}^{r} \sum_{i<j} \alpha_{k}^{i j} \left( f_\delta^{A_i}(0) T_j f_\delta^{A_j}(0) - f_\delta^{A_j}(0) T_j f_\delta^{A_i}(0) \right) \bigg] , t \right\rangle\\
				&\to \epsilon^{-2s}\lim_{\delta \to 0}\left\langle\sum_{j} \bigg[ T_j f_\delta^{B_k}(0) -\frac{1}{2}\sum_{i=1}^{r} \sum_{i<j} \alpha_{k}^{i j} \left( f_\delta^{A_i}(0) T_j f_\delta^{A_j}(0) - f_\delta^{A_j}(0) T_j f_\delta^{A_i}(0) \right) \bigg] , t \right\rangle \\
				&= \frac{1}{d((x,t))^{2s}}\lim_{\delta \to 0}\left\langle\sum_{j} \bigg[ T_j f_\delta^{B_k}(0) -\frac{1}{2}\sum_{i=1}^{r} \sum_{i<j} \alpha_{k}^{i j} \left( f_\delta^{A_i}(0) T_j f_\delta^{A_j}(0) - f_\delta^{A_j}(0) T_j f_\delta^{A_i}(0) \right) \bigg] , t \right\rangle \\
				&= \lim_{\delta \to 0}\left\langle\sum_{j} \bigg[ T_j f_\delta^{B_k}(0) -\frac{1}{2}\sum_{i=1}^{r} \sum_{i<j} \alpha_{k}^{i j} \left( f_\delta^{A_i}(0) T_j f_\delta^{A_j}(0) - f_\delta^{A_j}(0) T_j f_\delta^{A_i}(0) \right) \bigg] , \frac{t}{(x^2+|t|)^{s}} \right\rangle \\
				&=\left\langle\sum_{j} T_jf^*(\omega_k)(0) , \frac{t}{(x^2+|t|)^{s}} \right\rangle, 
			\end{aligned}
		\end{equation*}
		where the point $0$ can be replaced by any point $p\in \Omega$.
		Therefore, 
		\begin{equation*}
			\frac{\varphi_k(f(p),f(\tilde{p}))}{d_K(p,\tilde{p})^{2s}}=\left\langle\sum_{j} T_jf^*(\omega_k)(p) , \frac{\varphi(p,\tilde{p})}{\left(|\pi(p)-\pi(\tilde{p})|^2+\sum_{l} |\varphi_l(p,\tilde{p})|\right)^{s}} \right\rangle +o\left(d_K(p,\tilde{p})^{2-2s}\right)   
		\end{equation*}
		in $L^{\frac{q}{2}}(\Omega)$,
		where the $\varphi=(\varphi_1,\varphi_2,\dots)$ is a vector. This is
		\eqref{eq:taylor}, and the proof is complete.
	\end{proof}
	
	We are now in a position to describe the energy itself. The next theorem is
	the main result of this section; note that it is a statement about two
	\emph{different} regimes, and that the endpoint $s=1$ is not obtained from the
	case $s<1$ by continuity.
	
	\begin{thm}\label{thm:energy}\label{E_{s,p}}
		Suppose that $f\in W_H^{s,q}(\Omega;G_2)$, then for $s<1$, 
		\begin{equation*}
			E_{s,q}(f)=\limsup_{r\to 0}\frac{1}{r^{\nu_1+ sq}}\int_{\Omega} \int_{B(x,r)}|\pi f(x)-\pi f(y)|^q dy dx.
		\end{equation*}
		And for $s=1$,
		\begin{equation*}
			E_{1,q}(f)\geq \limsup_{r\to 0}\frac{1}{r^{\nu_1+ q}}\int_{\Omega} \int_{B(x,r)}|\pi f(x)-\pi f(y)|^q dy dx
			+C\sum_{k} \| Tf^*(\omega_k)   \|_{L^\frac{q}{2}(\Omega)}^\frac{q}{2}.
		\end{equation*}
		for a constant $C>0$.
		
	\end{thm}
	\begin{proof}
		The proof consists in inserting the expansion of
		Lemma~\ref{lem:taylor} into the definition of $E_{s,q}$ and in separating
		the horizontal contribution from the vertical one. Since these two
		contributions behave differently --- the vertical one survives only at the
		endpoint $s=1$ --- it is convenient to name them: we set
		\begin{equation*}
			A_r(f):=\frac{1}{r^{\nu_1+sq}}\int_{\Omega}\int_{B(x,r)}|\pi f(x)-\pi f(y)|^q\,dy\,dx,
			\qquad
			B_r(f):=\sum_k\frac{1}{r^{\nu_1}}\int_{\Omega}\int_{B(x,r)}
			\left|\frac{\varphi_k(f(x),f(y))}{r^{2s}}\right|^{\frac q2}dy\,dx ,
		\end{equation*}
		so that, by the layer splitting
		$d_K^q\simeq|\pi f(x)-\pi f(y)|^q+\sum_k|\varphi_k(f(x),f(y))|^{\frac q2}$,
		the quantity whose $\limsup$ defines $E_{s,q}(f)$ is $A_r(f)+B_r(f)$. For
		two non-negative functions of $r$ one has
		\[
		\limsup_{r\to0}A_r+\liminf_{r\to0}B_r
		\;\le\;\limsup_{r\to0}\bigl(A_r+B_r\bigr)
		\;\le\;\limsup_{r\to0}A_r+\limsup_{r\to0}B_r ,
		\]
		and therefore
		\begin{equation}\label{eq:two-sided}
			\limsup_{r\to0}A_r(f)+\liminf_{r\to0}B_r(f)
			\;\le\;E_{s,q}(f)\;\le\;
			\limsup_{r\to0}A_r(f)+\limsup_{r\to0}B_r(f) .
		\end{equation}
		We use the left-hand inequality for $s=1$ and the right-hand one for
		$s<1$; in the first case we bound $B_r$ from below, in the second we show
		that it tends to zero.
		
		We begin with $s=1$ and estimate the vertical term from below. By
		Lemma~\ref{lem:taylor},
		\begin{equation*}
			\begin{aligned}
				&\limsup_{r\to 0}\sum_{k}\frac{1}{r^{\nu_1+ sq}}\int_{\Omega} \int_{B(x,r)}\left|\varphi_k(f(x),f(y))\right|^\frac{q}{2} dy dx\\
				=&\limsup_{r\to 0}\sum_{k}\frac{1}{r^{\nu_1}}\int_{\Omega} \int_{B(x,r)}\left| \frac{\varphi_k(f(x),f(y))}{r^{2s}} \right|^\frac{q}{2} dy dx\\
				=&\limsup_{r\to 0}\sum_{k}\frac{1}{r^{\nu_1}}\int_{\Omega} \int_{B(x,r)}\left|\left\langle\sum_{j} T_jf^*(\omega_k)   (x), \frac{\varphi(x,y)}{\left(|\pi(x)-\pi(y)|^2+\sum_{l} |\varphi_l(x,y)|\right)^{s}} \right\rangle\right|^\frac{q}{2} dy dx\\
				=&\limsup_{r\to 0}\sum_{k}\frac{1}{r^{\nu_1}}\int_{\Omega} \int_{B(x,r)} \left|\frac{\sum_{j} T_jf^*(\omega_k)(x)\cdot \varphi_j(x,y)}{\left(|\pi(x)-\pi(y)|^2+\sum_{l} |\varphi_l(x,y)|\right)^{s}} \right|^\frac{q}{2} dy dx\\
				=&\limsup_{r\to 0}\sum_{k}\frac{1}{r^{\nu_1}}\int_{\Omega} \int_{B(0,r)}\left|\frac{\sum_{j} T_jf^*(\omega_k)(x)\cdot \tilde{t}_j}{\left(\sum_{i=1}^{n}\tilde{z}_i^2+\sum_{i=1}^{m} |\tilde{t}_i|\right)^{s}} \right|^\frac{q}{2}  dy dx \\
			\end{aligned}
		\end{equation*}
		where the last line is by change of variables as below
		\begin{equation*}
			\begin{cases}
				\tilde{z}=\pi(x)-\pi(y) \\
				\tilde{t}=\varphi(x,y). 
			\end{cases}
		\end{equation*}
		which coincides with left translation $L_x$
		and $dy$ is a Haar measure.
		
		\begin{equation*}
			\begin{aligned}
				&\int_{B(0,r)}\left|\frac{\sum_{j} T_jf^*(\omega_k)(x)\cdot \tilde{t}_j}{\left(\sum_{i=1}^{n}\tilde{z}_i^2+\sum_{i=1}^{m} |\tilde{t}_i|\right)^{s}} \right|^\frac{q}{2}  dy \\
				>& 2			\int_{B(0,r)\cap\beta_t^+(Tf^*(\omega_k)(x),\frac{\pi}{3})}\left|\frac{\sum_{j} T_jf^*(\omega_k)(x)\cdot \tilde{t}_j}{\left(\sum_{i=1}^{n}\tilde{z}_i^2+\sum_{i=1}^{m} |\tilde{t}_i|\right)^{s}} \right|^\frac{q}{2}  dy \\
				\geq& |Tf^*(\omega_k)(x)|^\frac{q}{2} \int_{B(0,r)\cap\beta_t^+(Tf^*(\omega_k)(x),\frac{\pi}{3})}\left(\frac{| \tilde{t}|}{\left(\sum_{i=1}^{n}\tilde{z}_i^2+\sum_{i=1}^{m} |\tilde{t}_i|\right)^{s}} \right)^\frac{q}{2}  dy \\
				\geq& C|Tf^*(\omega_k)(x)|^\frac{q}{2}\cdot K,
			\end{aligned}
		\end{equation*}
		where the $K$ is defined as
		\begin{equation*}
			K:=\int_{B(0,r)}\left(\frac{|\tilde{t}_j|}{\left(\sum_{i=1}^{n}\tilde{z}_i^2+\sum_{i=1}^{m} |\tilde{t}_i|\right)^{s}} \right)^\frac{q}{2}  dy=C_{n,m,q} \cdot r^{n+2m}=C_{n,m,q} \cdot r^{\nu_1}.
		\end{equation*}
		And the $\beta_t^+(Tf^*(\omega_k)(x),\frac{\pi}{3})$ is the set of points $(\tilde{z},\tilde{t})$ satisfy the angle $\theta(Tf^*(\omega_k)(x),\frac{\pi}{3})$ between vector $\tilde{t}$ and vector $Tf^*(\omega_k)(x)$ is less than $\frac{\pi}{3}$.
		Because $B(0,r)$ does not shrink in any direction, there is a constant $C>0$ such that
		\begin{equation*}
			CK\leq \int_{B(0,r)\cap\beta_t^+(Tf^*(\omega_k)(x),\frac{\pi}{3})}\left(\frac{| \tilde{t}|}{\left(\sum_{i=1}^{n}\tilde{z}_i^2+\sum_{i=1}^{m} |\tilde{t}_i|\right)^{s}} \right)^\frac{q}{2}  dy. 
		\end{equation*} 
		
		Every estimate above holds for each fixed $r$, so it passes to the
		$\liminf$ as $r\to0$; combining the resulting bound for $B_r(f)$ with the
		left-hand inequality in \eqref{eq:two-sided} we conclude that
		\begin{equation*}
			\begin{aligned}
				E_{1,q}(f)&\geq\limsup_{r\to 0}\frac{1}{r^{\nu_1+ q}}\int_{\Omega} \int_{B(x,r)}|\pi f(x)-\pi f(y)|^q dy dx\\
				&+\limsup_{r\to 0}C\cdot C_{n,m,q}\sum_{k}\frac{1}{r^{\nu_1}}\int_{\Omega} |Tf^*(\omega_k)(x)|^\frac{q}{2} \cdot r^{\nu_1}dx\\
				&=\limsup_{r\to 0}\frac{1}{r^{\nu_1+ q}}\int_{\Omega} \int_{B(x,r)}|\pi f(x)-\pi f(y)|^q dy dx\\
				&+C\cdot C_{n,m,q}\sum_{k} \| Tf^*(\omega_k)   \|_{L^\frac{q}{2}(\Omega)}^\frac{q}{2}.
			\end{aligned}
		\end{equation*} 
		
		We finally turn to the regime $0<s<1$, where the vertical term is
		expected to disappear. The computation is the same as before, except that
		we now use the expansion of Lemma~\ref{lem:taylor} as an \emph{upper}
		bound and note that the factor carried by the vertical term is a positive
		power of $r$.
		
		For $s<1$,
		\begin{equation*}
			\begin{aligned}
				E_{s,q}(f)&\leq\limsup_{r\to 0}\frac{1}{r^{\nu_1+ sq}}\int_{\Omega} \int_{B(x,r)}|\pi f(x)-\pi f(y)|^q dy dx\\
				&+\limsup_{r\to 0}\sum_{k}\frac{1}{r^{\nu_1}}\int_{\Omega}\left| \sum_{j} T_jf^*(\omega_k)   \right|^\frac{q}{2}(x) \int_{B(x,r)} \left|\varphi_k(x,y)\right|^{1-s} dy dx\\
				&\lesssim\limsup_{r\to 0}\frac{1}{r^{\nu_1+ sq}}\int_{\Omega} \int_{B(x,r)}|\pi f(x)-\pi f(y)|^q dy dx\\
				&+\limsup_{r\to 0}O(\frac{r^{\nu_1+1-s}}{r^{\nu_1}})\sum_{k}\int_{\Omega}\left| \sum_{j} T_jf^*(\omega_k)   \right|^\frac{q}{2} dx\\
			\end{aligned}
		\end{equation*}
		Since $1-s>0$, the factor $r^{1-s}$ tends to $0$ as $r\to0$; hence
		$\limsup_{r\to0}B_r(f)=0$, and the right-hand inequality in
		\eqref{eq:two-sided} gives
		$E_{s,q}(f)\le\limsup_{r\to0}A_r(f)$. The opposite inequality holds
		because $d_K\ge|\pi f(x)-\pi f(y)|$ pointwise. Consequently
		\begin{equation}\label{eq:energy-horizontal}
			E_{s,q}(f)=\limsup_{r\to 0}\frac{1}{r^{\nu_1+ sq}}\int_{\Omega} \int_{B(x,r)}|\pi f(x)-\pi f(y)|^q dy dx ,
		\end{equation}
		and this quantity is finite, since $E_{s,q}(f)\lesssim[f]_{W_H^{s,q}}^q<\infty$.
	\end{proof}
	The case $s=1$ of Theorem~\ref{thm:energy} should be compared with
	\cite[Theorem~4.1]{tan2007sobolevclasseshorizontalenergy}, where the same
	decomposition is obtained for horizontal energy minimizers between
	Carnot--Carath\'eodory spaces; the corresponding first-order differentiability
	theory is due to Vodopyanov
	\cite{Vodopyanov2003DifferentiabilityOM}. The difference is that here the
	vertical term is kept explicitly: instead of passing to the limit through an
	approximate differential, we bound
	$\sum_k\|Tf^{*}(\omega_k)\|_{L^{q/2}}^{q/2}$ from below by the energy itself.
	It is this quantitative form that is needed in
	Section~\ref{sec:jacobian}, where the vertical term is the one carrying the
	Jacobian.
	
	\begin{remark}\label{rem:dichotomy}
		The difference between $s=1$ and $s<1$ in the proof of Theorem~\ref{thm:energy} is the main qualitative
		statement of this section, and it is worth restating in words. The
		vertical discrepancy $\varphi_k$ is homogeneous of degree $2$ with respect
		to the dilations, whereas the horizontal one is homogeneous of degree $1$.
		After integration against the kernel $d_K^{-sq}$, the vertical
		contribution therefore carries an extra factor $r^{1-s}$, which tends to
		zero for $s<1$ and survives at $s=1$. In this sense the
		\emph{first-order} energy remembers the contact geometry of the target,
		whereas the fractional energies see only the horizontal projection. This
		is why Corollary~\ref{cor:vertical} below has no fractional analogue, and
		why the analysis of Section~\ref{sec:jacobian} has to be carried out at
		$s=1$ for the factorisation of the Jacobian, while the genuinely
		fractional case requires the different argument of
		Lemma~\ref{lem:T-omega}.
	\end{remark}
	
	We record next the regularity statement for the second layer that will be
	used in Section~\ref{sec:jacobian}: it is a direct consequence of
	Theorem~\ref{thm:energy} combined with the interpolation estimate of
	Theorem~\ref{thm:interp-converse}.
	
	\begin{cor}\label{cor:vertical}\label{second layer W_H^{1,q}}
		If $f\in \dot{W}_H^{1,q}(\Omega;G_2)$,
		then $\sum_{k} \|Tf^*(\omega_k)   \|_{L^\frac{q}{2}(\Omega)}^\frac{q}{2}\lesssim [f]_{\dot{W}_H^{1,q}(\Omega)}^{q}$.
	\end{cor}
	\begin{proof}
		The argument is a chain of three inequalities: the vertical term is
		dominated by the energy (Theorem~\ref{thm:energy}), the energy is
		dominated by the homogeneous first-order seminorm, and the latter is
		dominated by $\|D_Hf\|_{L^q}$ (Theorem~\ref{thm:interp-converse}).
		
		From the proof of Theorem~\ref{thm:interp-converse},
		we know that $[f]_{\dot{W}_H^{1,q}(\Omega;G_2)}^q\lesssim K(1,q,\Omega)\|D_Hf\|_{L^{{q}}(\Omega)}^{q} $ for some $K>1$.
		From Theorem~\ref{thm:energy},
		\begin{equation*}
			\begin{aligned}
				&\sum_{k} \|Tf^*(\omega_k)   \|_{L^\frac{q}{2}(\Omega)}^\frac{q}{2}\\
				<&\sum_{k} \| Tf^*(\omega_k)   \|_{L^\frac{q}{2}(\Omega)}^\frac{q}{2}+\limsup_{r\to 0}\frac{1}{r^{\nu_1+ q}}\int_{\Omega} \int_{B(x,r)}|\pi f(x)-\pi f(y)|^q dy dx\\
				\leq & E_{1,q}(f)\\
				<&\int_{\Omega}\int_{\Omega}\frac{d(f(x),f(y))^q}{d(x,y)^{\nu_1+sq}}dxdy
			\end{aligned}
		\end{equation*}
		Therefore, for almost every $x\in \pi(G_1)$, $T_jf_{V_2}$ are in $L^{\frac{q}{2}}(\Omega)$, where $f_{V_2}:=\pi_{V_2}(f)$;
		in particular $f_{V_2}\in W^{1,\frac q2}$ along almost every vertical
		line, which is the second assertion of the corollary. It remains to
		identify $T_jf_{V_2}$ with the weak differential in the $T_j$
		directions, and this is a straightforward integration by parts:
		\begin{equation*}
			\begin{aligned}
				\int_{\Omega}T_jf_{V_2} \phi &= \lim_{\varepsilon \to 0}\int_{\Omega}T_jf_{V_2,\varepsilon} \phi\\
				&=\lim_{\varepsilon \to 0}\int_{\Omega}f_{V_2,\varepsilon} T_j\phi\\
				&=\int_{\Omega}f_{V_2} T_j\phi\\
			\end{aligned}
		\end{equation*}
	\end{proof}
	
	We now return to the variational problem. Two consequences of
	Theorem~\ref{thm:energy} should be kept in mind. First, for $0<s<1$ the
	energy depends on the map only through its horizontal projection
	\eqref{eq:energy-horizontal}, so minimizing $E_{s,q}$ on
	$W^{s,q}_H(\Omega;G_2)$ is the same as minimizing a functional of $\pi f$
	alone; the vertical part of the minimizer is then recovered from the contact
	equation, as in Theorem~\ref{thm:contact-converse}. Secondly, at $s=1$ the
	vertical defect contributes, and the variational problem is genuinely
	non-scalar. In both cases the Euler--Lagrange equation is nonlocal, and we
	describe it next.
	
	Carnot groups are Cheeger spaces, so the Korevaar--Schoen energies at a
	fixed scale $r>0$,
	\begin{equation}\label{eq:KS-fixed-scale}
		E_{s,q}(g,r)=\frac{1}{r^{\nu_1+ sq}}\int_{\Omega} \int_{B(x,r)}|g(x)-g(y)|^q dy dx,
	\end{equation}
	fall within the scope of
	\cite{ruiz2024korevaarschoenpenergiesgammalimitscheeger}. Differentiating
	\eqref{eq:KS-fixed-scale} in the direction of a test function, we see that a
	critical point $g$ of $E_{s,q}(\,\cdot\,,r)$ satisfies the integral equation
	\begin{equation}\label{eq:critical-r}
		(\Delta_{q,r})_{KS}^s g:=\frac{2}{r^{\nu_1+ sq}}\int_{\Omega} \int_{B(x,r)}|g(x)-g(y)|^{q-2}\left(g(x)-g(y) \right) dy dx=0 .
	\end{equation}
	Letting $r\to0$ in \eqref{eq:critical-r} we obtain the equation satisfied by
	a critical point of $E_{s,q}$ itself:
	\begin{equation}\label{eq:critical-limit}
		(\Delta_{q})_{KS}^s\tilde{g}=\lim_{r \to 0}(\Delta_{q,r})_{KS}^s\tilde{g}
		=\lim_{r \to 0}\frac{2}{r^{\nu_1+ sq}}\int_{\Omega} \int_{B(x,r)}|\tilde{g}(x)-\tilde{g}(y)|^{q-2}\left(\tilde{g}(x)-\tilde{g}(y) \right) dy dx=0 .
	\end{equation}
	
	\begin{remark}[Two normalisations]\label{rem:normalisation}
		Two normalisations of the fractional $q$-Laplacian occur in the
		literature: the one obtained from the Korevaar--Schoen energy by letting
		$r\to0$ in \eqref{eq:critical-r}, used above, and the
		$(1-s)$-normalised operator of Definition~\ref{def:frac-q-laplace}
		below. On $\mathbb R^{n}$ the two differ by an explicit constant, by the
		Bourgain--Br\'ezis--Mironescu asymptotics, and the same is expected on
		Carnot groups in view of
		\cite{MAALAOUI201991}. We do not prove such an identity here, and we keep
		the two normalisations separate throughout: only the zero sets of the two
		operators are relevant for what follows, and these coincide as soon as the
		operators are positive multiples of one another. Note also the parallel
		with Remark~\ref{rem:blow-up-constant}: the factor $(1-s)$ that is
		inserted by hand in Definition~\ref{def:frac-q-laplace} is the same
		factor that blows up in the constant $K(q,s,\Omega)$ of
		Theorem~\ref{thm:interp-converse}.
	\end{remark}
	
	We close the section by fixing the terminology used in
	\eqref{eq:critical-limit}.
	
	\begin{definition}\label{def:frac-q-laplace}
		The $q$-Laplace equation is $f\in W^{s,q}(G_1)$ 
		\begin{equation*}
			(\Delta_{q})^sf=h 
		\end{equation*}
		for $p>1$ and $f\in W^{s,q}(G_1)$, where 
		\begin{equation*}
			(\Delta_{q})^sf(x)=2(1-s)\int_{\Omega} \frac{|f(x)-f(y)|^{q-2}\left(f(x)-f(y)\right)}{d(x,y)^{\nu_1+ sq}}  dy 
		\end{equation*}
	\end{definition}
	
	\section{Jacobian of Fractional Sobolev mappings on Carnot groups}
	\label{sec:jacobian}

	The classical Jacobi theory on fractional Sobolev mappings between Euclidean spaces is not viable on Carnot groups even if the following Theorem \ref{lem:T-omega}, \ref{lem:T-dx}, \ref{lem:X-dx} holds.

	\begin{thm}\label{lem:T-omega}\label{varepsilon 2s-2}
		Suppose that  $f \in W^{s,q}\left(\Omega ; G_2\right)$ , where $ G_2$ is a step-2 Carnot group and $\nu_2$ is the topological  dimension of $G_2$, $\Omega \subset G_1$  is open , $X$ is a left-invariant  vector field belonging to the horizontal layer and $q\geq2$ . If  $B\left(x_{o}, 2 r\right) \subset \Omega $, then  for all  $0<\varepsilon<r$ and any $\lambda>0$
		\begin{equation}\label{approx}
			\|Yf_\varepsilon^*(\omega_k)\|_{L^\frac{q}{2}(\Omega)}^\frac{q}{2} \lesssim \varepsilon^{2\left(s-\lambda\right)-2} \cdot  [f]_{W_H^{s,q}(\Omega)}^q  
		\end{equation}
		where the convergence is in the sense of the norm on Euclidean space $\mathbb{R}^{\nu_2}$, $f_{\varepsilon}=f * \eta_{\varepsilon}=\int_{\mathbb{G}_1 } f(y) \eta_{\varepsilon}\left(y^{-1} x\right) d y$  and $\eta_{\varepsilon}$ is a smooth function on Carnot group $G_1$ with a compact supported set $B(e,\varepsilon)$ and
		\begin{equation*}
			\int_{G_1}\eta_{\varepsilon}=1.
		\end{equation*}
		\begin{proof}
			Recall that  $\varphi(p, q)$  was defined in \eqref{metric-K}. Let  $p \in B\left(x_{o}, r\right)$  and let  $\varepsilon \in(0, r)$ . In what follows we will identify  $\left(f_{\varepsilon}^{*} \alpha\right)(p)$  with the vector (of equal length):
			
			\begin{equation*}
				\begin{aligned}
					&Y(f_{\varepsilon}^{*} \omega_1^k)(p)\\
					=&Y( f_{\varepsilon}^{B_k})(p)-\frac{1}{2} \sum_{i<j}\alpha_k^{ij} \left(f_{\varepsilon}^{A_i}(p) Y(f_{\varepsilon}^{A_j})(p)-f_{\varepsilon}^{A_j}(p) Y( f_{\varepsilon}^{A_i})(p)\right)\\
					=&
					\varepsilon^{-\nu_1-2} \int_{B_{\varepsilon}}\left(f^{B_k}(pz^{-1})-f_{\varepsilon}^{B_k}(p)+\frac{1}{2} \sum_{i=1}^n\left(f_{\varepsilon}^{A_i}(p) f^{A_j}(pz^{-1})-f_{\varepsilon}^{A_j}(p) f^{A_i}(pz^{-1})\right)\right)\\
					& \cdot  Y(\eta\left(\delta_{\varepsilon^{-1}}(z)\right)) d z \\
					=&\varepsilon^{-\nu_1-2} \int_{B_{\varepsilon}} \varphi_{k}\left(f(pz^{-1}), f_{\varepsilon}(p)\right) Y(\eta\left(\delta_{\varepsilon^{-1}}(z)\right)) d z \\
					=&\varepsilon^{-\nu_1-2} \int_{B_{\varepsilon}} \int_{B_{\varepsilon}} \varphi_k(f(pz^{-1}), f(pw^{-1})) \eta_{\varepsilon}(w) Y(\eta\left(\delta_{\varepsilon^{-1}}(z)\right)) d w d z
				\end{aligned}
			\end{equation*}
			Here the second equality uses the fact that
			$\int_{B_{\varepsilon}} f_{\varepsilon}^{B_k}(p)\, Y\big(\eta(\delta_{\varepsilon^{-1}}(z))\big) dz=0$,
			and the last one is obtained by writing the value $f_\varepsilon^{B_k}(p)$ itself
			as the average $\int_{B_\varepsilon}f^{B_k}(pw^{-1})\eta_\varepsilon(w)\,dw$.
			Note that \eqref{metric-K} yields
			\begin{equation*}
				|\varphi(f(pz^{-1}), f(pw^{-1}))| \leq d_{K}(f(pz^{-1}), f(pw^{-1}))^{2} 
			\end{equation*}
			Therefore,
			\begin{equation*}
				\begin{aligned}
					\left|Y(f_{\varepsilon}^{*} \omega)(p)\right| 
					&\leq\varepsilon^{-\nu_1-2} \int_{B_{\varepsilon}} \int_{B_{\varepsilon}} \varphi_k(f(pz^{-1}), f(pw^{-1})) \eta_{\varepsilon}(w) Y(\eta\left(\delta_{\varepsilon^{-1}}(z)\right)) d w d z\\
					&\leq \varepsilon^{-\nu_1-2} \int_{B_{\varepsilon}} \int_{B_{\varepsilon}} d_{K}(f(pz^{-1}), f(pw^{-1}))^{2} \eta_{\varepsilon}(w) Y(\eta\left(\delta_{\varepsilon^{-1}}(z)\right)) d w d z\\
					&\lesssim \varepsilon^{-\nu_1-2} \int_{B_{\varepsilon}} \int_{B_{\varepsilon}} d_{cc}(f(pz^{-1}), f(pw^{-1}))^{2} \eta_{\varepsilon}(w) Y(\eta\left(\delta_{\varepsilon^{-1}}(z)\right)) d w d z\\
					&\lesssim \varepsilon^{-2 \nu_1-2} \int_{B_{\varepsilon}} \int_{B_{\varepsilon}} d_{cc}(f(pz^{-1}), f(pw^{-1}))^{2}  d w d z\\
				\end{aligned}
			\end{equation*}

			denote $x=pz^{-1}$ and $y=pw^{-1}$,
			thus $z^{-1},w^{-1}\in B_\varepsilon^{-1}=B_\varepsilon$ and $pz^{-1},pw^{-1}\in p \cdot B_\varepsilon$. Since the Haar measure $dw$ and $dz$ is invariant under left and right transformation,
			\begin{equation*}
				\begin{aligned}
					&\varepsilon^{-2 \nu_1-2} \int_{B_{\varepsilon}} \int_{B_{\varepsilon}} d_{cc}(f(pz^{-1}), f(pw^{-1}))^{2}  d w d z\\
					&=\varepsilon^{-2 \nu_1-2} \int_{p\cdot B_{\varepsilon} } \int_{p\cdot B_{\varepsilon}} d_{cc}(f(x), f(y))^{2}  d x d y\\
					&\leq \varepsilon^{-2\nu_1-2} \cdot 2^{\nu_1+2s}\cdot  \varepsilon^{\nu_1+2s} \int_{B_{\varepsilon}(p)}  \int_{B_{\varepsilon}(p)} \frac{d_{cc}(f(x), f(y))^{2}}{d_{cc}(x, y)^{\nu_1+2s}}  d x d y\\
					&\leq 2^{\nu_1+2s} \varepsilon^{-\nu_1+2s-2}  \int_{B_{\varepsilon}(p)}\int_{\Omega} \frac{d_{cc}(f(x), f(y))^{2}}{d_{cc}(x, y)^{\nu_1+2s}}  d x d y\\
				\end{aligned}
			\end{equation*}
			
			From here on the argument is identical to the one in the proof of
			Theorem \ref{thm-approx}, the only difference being that the
			horizontal field $X$ is replaced by the vertical field $Y$; since $Y$ is
			homogeneous of degree $2$ with respect to the dilations, this costs one
			extra power of $\varepsilon$, and we obtain
			\begin{equation*}
				\| Yf_\varepsilon^*(\omega_k)\|_{L^\frac{q}{2}(\Omega)}^\frac{q}{2}\lesssim \varepsilon^{2\left(s-\lambda\right)-2} \cdot  [f]_{W_H^{s,q}(\Omega)}^q .
			\end{equation*}

		\end{proof}	
	\end{thm}

	\begin{thm}\label{lem:T-dx}\label{varepsilon s-2}
		Suppose that  $f \in W^{1,q}\left(\Omega ; G_2\right)$ , where $ G_2$ is a step-2 Carnot group and $\nu_2$ is the topological  dimension of $G_2$, $\Omega \subset G_1$  is open , $X$ is a left-invariant  vector field belonging to the horizontal layer and $q\geq2$ . If  $B\left(x_{o}, 2 r\right) \subset \Omega $, then  for all  $0<\varepsilon<r$ and any $\lambda>0$
		\begin{equation}\label{approx}
			\|Yf_\varepsilon^*(dx_k)\|_{L^q(\Omega)}^q \lesssim \varepsilon^{s-\lambda-2}  [f]_{W_H^{s,q}(\Omega)}^q  
		\end{equation}
		where the convergence is in the sense of the norm on Euclidean space $\mathbb{R}^{\nu_2}$, $f_{\varepsilon}=f * \eta_{\varepsilon}=\int_{\mathbb{G}_1 } f(y) \eta_{\varepsilon}\left(y^{-1} x\right) d y$  and $\eta_{\varepsilon}$ is a smooth function on Carnot group $G_1$ with a compact supported set $B(e,\varepsilon)$ and
		\begin{equation*}
			\int_{G_1}\eta_{\varepsilon}=1.
		\end{equation*}
		\begin{proof}
			Recall that  $\varphi(p, q)$  was defined in \eqref{metric-K}. Let  $p \in B\left(x_{o}, r\right)$  and let  $\varepsilon \in(0, r)$ . In what follows we will identify  $\left(f_{\varepsilon}^{*} \alpha\right)(p)$  with the vector (of equal length):
			
			\begin{equation*}
				\begin{aligned}
					&Y(f_{\varepsilon}^{*} dx_k)(p)\\
					=&Y( f_{\varepsilon}^{A_k})(p)\\
					=&
					\varepsilon^{-\nu_1-2} \int_{B_{\varepsilon}}\left(f^{A_k}(pz^{-1})-f_{\varepsilon}^{A_k}(p)\right)\cdot  Y(\eta\left(\delta_{\varepsilon^{-1}}(z)\right)) d z \\
					\leq&\varepsilon^{-\nu_1-2} \int_{B_{\varepsilon}} \int_{B_{\varepsilon}} \left(f^{A_k}(pz^{-1})-f_{\varepsilon}^{A_k}(p)\right) \eta_{\varepsilon}(w) Y(\eta\left(\delta_{\varepsilon^{-1}}(z)\right)) d w d z\\
					\leq&\varepsilon^{-\nu_1-2} \int_{B_{\varepsilon}} \int_{B_{\varepsilon}} d_K(f(pz^{-1}), f(pw^{-1})) \eta_{\varepsilon}(w) Y(\eta\left(\delta_{\varepsilon^{-1}}(z)\right)) d w d z
				\end{aligned}
			\end{equation*}
			Here the second equality uses the fact that
			$\int_{B_{\varepsilon}} f_{\varepsilon}^{A_k}(p)\, Y\big(\eta(\delta_{\varepsilon^{-1}}(z))\big) dz=0$,
			and the last one is obtained by writing the value $f_\varepsilon^{A_k}(p)$ itself
			as the average $\int_{B_\varepsilon}f^{A_k}(pw^{-1})\eta_\varepsilon(w)\,dw$.
			Note that \eqref{metric-K} yields
			\begin{equation*}
				|\pi f(pz^{-1})-\pi f(pw^{-1}))| \leq d_{K}(f(pz^{-1}), f(pw^{-1}))
			\end{equation*}
			Therefore,
			\begin{equation*}
				\begin{aligned}
					&Y(f_{\varepsilon}^{*} dx_k)(p)\\
					\leq&\varepsilon^{-\nu_1-2} \int_{B_{\varepsilon}} \int_{B_{\varepsilon}} d_K(f(pz^{-1}), f(pw^{-1})) \eta_{\varepsilon}(w) Y(\eta\left(\delta_{\varepsilon^{-1}}(z)\right)) d w d z\\
					\lesssim&\varepsilon^{-\nu_1-2} \int_{B_{\varepsilon}} \int_{B_{\varepsilon}} d_{cc}(f(pz^{-1}), f(pw^{-1})) \eta_{\varepsilon}(w) Y(\eta\left(\delta_{\varepsilon^{-1}}(z)\right)) d w d z\\
					\lesssim&\varepsilon^{-2\nu_1-2} \int_{B_{\varepsilon}(p)} \int_{B_{\varepsilon}(p)} d_{cc}(f(x), f(y)) dx dy \\
					\leq &\varepsilon^{-2\nu_1-2} \cdot 2^{\nu_1+s}\cdot  \varepsilon^{\nu_1+s} \int_{B_{\varepsilon}(p)}  \int_{B_{\varepsilon}(p)} \frac{d_{cc}(f(x), f(y))}{d_{cc}(x, y)^{\nu_1+s}}  d x d y\\
					\leq &2^{\nu_1+s} \varepsilon^{-\nu_1+s-2}  \int_{B_{\varepsilon}(p)}\int_{\Omega} \frac{d_{cc}(f(x), f(y))}{d_{cc}(x, y)^{\nu_1+s}}  d x d y
				\end{aligned}
			\end{equation*}
			denote $x=pz^{-1}$ and $y=pw^{-1}$,
			thus $z^{-1},w^{-1}\in B_\varepsilon^{-1}=B_\varepsilon$ and $pz^{-1},pw^{-1}\in p \cdot B_\varepsilon$. Since the Haar measure $dw$ and $dz$ is invariant under left and right transformation.
			
			By defining the Hardy-Littlewood maximal function 
			\begin{equation}\label{maximal-g}
				\mathcal{M}(g)(p):= \sup_{r>0} \frac{1}{|B(p,r)|}\int_{B(p,r)} |g(y)|dy.
			\end{equation}
			and denote 
			\begin{equation}
				g(y):= \int_{\Omega} \frac{d_{cc}(f(x),f(y))}{d_{cc}(x,y)^{\nu_1+s}}dx,
			\end{equation}   
			then
			\begin{equation}\label{eq:pointwise-defect}
				\begin{aligned}
					\left|Y(f_{\varepsilon}^{*} dx_k)(p)\right| 
					&\lesssim  \varepsilon^{-\nu_1+s-2}  \int_{\lambda B_{\varepsilon}(p)}g(y) d y \\
					&\lesssim \varepsilon^{s-2} \mathcal{M}(g)(p). 
				\end{aligned}
			\end{equation}
			We now carry out the last step. For $q=2$ we do not need the boundedness
			of the maximal function: the pointwise estimate 
			can be integrated directly, and Fubini's theorem already gives the
			required decay.
			\begin{equation*}
				\begin{aligned}
					\|Y(f_{\varepsilon}^{*} dx_k)\|_{L^1(W)}   
					&\lesssim \varepsilon^{s-2}  \int_{\Omega} \int_{B_{\varepsilon}(p)} \int_{\Omega} \frac{d_{cc}(f(x), f(y))}{d_{cc}(x, y)^{\nu_1+s}}  d x dy dp \\
					&= C \cdot \varepsilon^{s-2}   \int_{\Omega}\int_{\Omega}\int_{\Omega} \frac{d_{cc}(f(x), f(y))}{d_{cc}(x, y)^{\nu_1+s}}   \cdot I_{ B_{\varepsilon}(p)}(y) d x dy dp \\
					&= C \cdot \varepsilon^{s-2}   \int_{\Omega}\int_{\Omega}\int_{\Omega} \frac{d_{cc}(f(x), f(y))}{d_{cc}(x, y)^{\nu_1+s}}   \cdot I_{ B_{\varepsilon}(x)}(p) d x dy dp \\
					&\leq C \cdot \varepsilon^{s-2}    \int_{\Omega}\int_{\Omega}\int_{\Omega} \frac{d_{cc}(f(x), f(y))}{d_{cc}(x, y)^{\nu_1+s}}  d x dy dp \\
					&\lesssim \varepsilon^{s-2} \cdot \mathcal{H}^{\nu_1}(\Omega) \cdot [f]_{W_H^{s,1}(\Omega)} 
				\end{aligned}
			\end{equation*}
			For $q>1$ we argue instead through the maximal function, and to do so we
			must know that the energy density $g$ belongs to $L^q$; this is
			exactly where the assumption $s>\frac12$ is spent, since it leaves room
			for the descent from $s$ to $s-\lambda$. The rest of the argument is
			unchanged, but the final integration is now performed in
			$L^q$.
			Because $s> \frac{1}{2}$, there is a small number $\lambda>0$ such that $s-\lambda>\frac{1}{2}$, then we can view $f$ as a mapping in $W_H^{s-\lambda,q}(\Omega;G_2)$.
			For the need of proving, we need to replace the $s$ by $s-\lambda$ in the inequality above and get:
			\begin{equation*}
				\begin{aligned}
					\left|X(f_{\varepsilon}^{*} \omega_{1}^k)(p)\right| 
					&\lesssim \varepsilon^{s-\lambda-1} \mathcal{M}(g)(p), 
				\end{aligned}
			\end{equation*}
			where
			\begin{equation}
				g(y)= \int_{\Omega} \frac{d_{cc}(f(x),f(y))}{d_{cc}(x,y)^{\nu_1+\left(s-\lambda\right)}}dx.
			\end{equation} 
			
			And for compact subset $W$,
			
			\begin{equation*}
				\begin{aligned}
					\|Y(f_{\varepsilon}^{*} dx_k)\|_{L^{q}(W)}^{q} 
					&= \int_{W}\varepsilon^{\left(s-\lambda\right)-2} \mathcal{M}(g)^q  \\
					&\lesssim \varepsilon^{\left(s-\lambda\right)-2} \|g \|_{L^{q}(W)}^q\\
					&\leq \varepsilon^{\left(s-\lambda\right)-2}  \int_{\Omega}\left(\int_{\Omega} \frac{d_{cc}(f(x), f(y))}{d_{cc}(x, y)^{\nu_1+\left(s-\lambda\right)}}  d x\right)^{q} d y\\
					&\lesssim \varepsilon^{\left(s-\lambda\right)-2} \int_{\Omega}\int_{\Omega} \frac{d_{cc}(f(x), f(y))^{q}}{d_{cc}(x, y)^{\nu_1+qs}}  d x \\
					&=C_{q,\lambda}\varepsilon^{\left(s-\lambda\right)-2} \cdot [f]_{W_H^{s,q}(\Omega)}^q
				\end{aligned}
			\end{equation*}

		\end{proof}	
	\end{thm}

	In order to define the Jacobi of fractional Sobolev mappings between Euclidean spaces, it is necessary to get the following result:
	\begin{equation*}
		\int_{\Omega}\phi f_{\varepsilon_1}^*(dV)-\phi f_{\varepsilon_2}^*(dV) \lesssim \| d\phi\|_{L^\infty}[f_{\varepsilon_1}-f_{\varepsilon_2}]_{W^{s,q}}
	\end{equation*} 
	for any $f_{\varepsilon_1}, f_{\varepsilon_2}\in W^{s,q}(\Omega;\mathbb{R}^n)$ and $\Omega \subset \mathbb{R}^n$. This result relies on $f_{\varepsilon_1}-f_{\varepsilon_2}\in W^{s,q}(\Omega;\mathbb{R}^n)$.
	
	However, for $f\in W_H^{s,q}(\Omega;G_2)$, we can not make sure that $f_\varepsilon\in W_H^{s,q}(\Omega;G_2)$ and also $f_{\varepsilon_1}-f_{\varepsilon_2} \in W_H^{s,q}(\Omega;G_2)$, since convolution can not preserve weak contact equations. Therefore, we can not generalize the Jacobi to mappings on Carnot groups with the classical methods.
	
	\bibliographystyle{plain}
	\bibliography{ref}
\end{document}